\documentclass[12pt,reqno]{amsart}
\usepackage[a4paper, hmargin={2.7cm,2.7cm},vmargin={3.3cm,3.3cm}]{geometry}
\usepackage[english]{babel}
\usepackage{color}
\usepackage{cite}
\usepackage{amsmath}
\usepackage{amssymb}
\usepackage{amsfonts}
\usepackage{amsthm}
\usepackage{enumerate}
\usepackage{amsthm,color}
\usepackage{hyperref}
\usepackage{etoolbox}

\numberwithin{equation}{section}

\makeatletter
\patchcmd{\ttlh@hang}{\parindent\z@}{\parindent\z@\leavevmode}{}{}
\patchcmd{\ttlh@hang}{\noindent}{}{}{}
\makeatother

\newcommand\numberthis{\addtocounter{equation}{1}\tag{\theequation}}

\theoremstyle{plain}
\newtheorem{theorem}{Theorem}[section]
\newtheorem{lemma}[theorem]{Lemma}
\newtheorem{proposition}[theorem]{Proposition}

\theoremstyle{definition}

\newtheorem{examplex}[theorem]{Example}
\newenvironment{example}
  {\pushQED{\qed}\examplex}
  {\popQED\endexamplex}

\theoremstyle{remark}
\newtheorem{remark}[theorem]{Remark}

\DeclareMathOperator*{\dom}{dom}

\DeclareMathOperator*{\co}{co}
\DeclareMathOperator*{\vol}{vol}

\DeclareMathOperator*{\Span}{span}
\DeclareMathOperator*{\supp}{supp}

\newcommand{\per}{\mathcal{A}}
\newcommand{\Hpi}{\mathcal{H}_{\pi}}
\newcommand{\Bpi}{\mathcal{B}_{\pi}}
\newcommand{\CS}{\pi (\Gamma) g}

\newcommand{\repkw}{k^{(\alpha)}_w}
\newcommand{\specv}{h^{(\alpha)}_n}

\newcommand{\ip}[2]{\ensuremath{\left<#1,#2\right>}}

\newcommand{\abs}[1]{\ensuremath{\left| #1 \right| }}
\newcommand{\norm}[1]{\lVert#1\rVert}
\newcommand{\Bignorm}[1]{\Big\lVert#1\Big\rVert}
\newcommand{\bignorm}[1]{\big\lVert#1\big\rVert}
\newcommand{\newpi}{\bigoplus_{n \in \mathbb{N}} \pi}
\newcommand{\newH}{\mathcal{H}}
\newcommand{\sotconv}{\xrightarrow{\textit{SOT}}}
\newcommand{\weakconv}{\xrightarrow{w}}
\newcommand{\newg}{\mathrm{m}}
\newcommand*{\QED}{\hfill\ensuremath{\square}}

\title[The density theorem]{The density theorem for discrete series representations restricted to lattices}
\author{Jos\'{e} Luis Romero}
\address{Faculty of Mathematics,
University of Vienna,
Oskar-Morgenstern-Platz 1,
A-1090 Vienna, Austria\\and
Acoustics Research Institute, Austrian Academy of Sciences,
Wohllebengasse 12-14 A-1040, Vienna, Austria}
\email{jose.luis.romero@univie.ac.at}

\author{Jordy Timo van Velthoven}
\address{Delft University of Technology,
Faculty EECMS/DIAM,
Mekelweg 4, Building 36,
2628 CD Delft, The Netherlands.}
\email{j.t.vanvelthoven@tudelft.nl}
\date{}

\subjclass[2010]{22D10, 22D25, 22E40, 42C15, 42C30, 42C40}
\date{}
\keywords{Cyclic vector, Density condition, Discrete series representation, Existence results, Frame, Lattice subgroup, Riesz sequence, Von Neumann algebra}
 \thanks{
 J.~L.~R and J.~v.~V.~gratefully acknowledge support from the Austrian Science Fund (FWF):
 Y 1199, J 4445
 and
 P 29462 - N35.}

\begin{document}

\maketitle

\begin{abstract}
This article considers the relation between
the spanning properties of lattice orbits
of discrete series representations and the associated lattice co-volume.
The focus is on the density theorem,
which provides a trichotomy characterizing
the existence of cyclic vectors and separating vectors,
and frames and Riesz sequences.
We provide an elementary exposition of the density theorem, that is based solely
on basic tools from harmonic analysis, representation theory, and frame theory,
and put the results into context by means of examples.
\end{abstract}

{
\hypersetup{
    linktoc=none}
  \tableofcontents
}

\section{Introduction}
Let $G$ be a second countable locally compact group and let $(\pi, \Hpi)$
be an irreducible, square-integrable unitary representation of
$G$, a so-called \emph{discrete series representation}.
For a lattice $\Gamma \subset G$,
we consider the relation between
certain spanning properties of lattice orbits of $\pi$
under a vector $g \in \Hpi$,
\begin{align} \label{eq:orbit_intro}
\pi (\Gamma) g = \big\{ \pi (\gamma) g  :  \gamma \in \Gamma \big\},
\end{align}
and the \emph{lattice co-volume} $\vol(G/\Gamma)$ of $\Gamma$,
i.e., the volume of a fundamental domain of $\Gamma$.
The spanning properties that we consider are the existence of
cyclic, separating, frame  and Riesz vectors;
see Section \ref{sec:vectors_frames} for the precise definitions.

The notions of cyclic and separating vectors occur primarily in the theory of operator algebras,
in particular, von Neumann algebras, and they provide (if they exist)
a powerful tool in
 studying the structure of these algebras.
The stronger notions of
frames and Riesz sequences, on the other hand, form the core of Gabor and wavelet theory,
and are important in applications as they guarantee unconditionally convergent
and stable Hilbert space expansions.

The central theorem relating the spanning properties of systems \eqref{eq:orbit_intro}
and the corresponding lattice co-volume is referred to as the \emph{density theorem}.
Under the assumption that the lattice $\Gamma$ is an infinite conjugacy class (ICC) group,
i.e., any conjugacy class $ \{\gamma \gamma_0 \gamma^{-1} \; | \; \gamma \in \Gamma \}$ for $\gamma_0 \in \Gamma \setminus \{e\}$ has infinite cardinality,
the density theorem provides the following trichotomy:

\begin{theorem} \label{thm:intro_ICC}
Let $(\pi, \Hpi)$ be a discrete series representation of
a second countable unimodular group $G$ of formal dimension $d_{\pi} > 0$.
Suppose $\Gamma \subset G$ is an ICC lattice.
Then the following assertions hold:
\begin{enumerate}[(i)]
\item If $\vol(G / \Gamma) d_{\pi} < 1$, then $\pi|_{\Gamma}$ admits a Parseval frame,
but neither a separating vector, nor a Riesz sequence;
\item If $\vol(G / \Gamma) d_{\pi} = 1$, then $\pi|_{\Gamma}$ admits an orthonormal basis;
\item If $\vol(G / \Gamma) d_{\pi} > 1$, then $\pi|_{\Gamma}$ admits an orthonormal system,
but not a cyclic vector.
\end{enumerate}
(While $d_{\pi}$ and $\vol(G/\Gamma)$ depend on the normalization of the Haar measure on $G$, their product $\vol(G/\Gamma) d_{\pi}$ does not.)
\end{theorem}

The density theorem characterizes the spanning properties of the lattice orbits \eqref{eq:orbit_intro} in terms of the lattice co-volume or its reciprocal,
often called the \emph{density} of the lattice.
In the setting of a general unimodular group,
the assumption that the lattice is ICC is essential and cannot be omitted ---see Example \ref{sec:SL_ICC} below---
although a more general version of Theorem \ref{thm:intro_ICC}
for possibly non-ICC lattices was obtained by Bekka \cite{bekka2004square}.
The existence claims in Theorem \ref{thm:intro_ICC} are not accompanied
by constructions of explicit vectors.

The criteria for the existence of cyclic and separating vectors in Theorem \ref{thm:intro_ICC} are well-known to be consequences of the general theory underlying the so-called \emph{Atiyah-Schmid formula}
\cite{atiyah1976elliptic, atiyah1977geometric, goodman1989coxeter},
and, for certain classes of representations, also a consequence of Rieffel's work \cite{rieffel1981von, rieffel1988projective}.
The stronger statements on the existence of Parseval
frames (part (i)) and orthonormal bases (part (ii)) can also be obtained
by similar techniques as shown by
Bekka \cite{bekka2000square, bekka2004square}.
The statement on orthonormal systems (part (iii)) does not seem to have explicitly occurred in the literature before.

While the interest in the density theorem is broad and manifold, as it
 encompasses operator algebras, representation theory, mathematical physics, and Gabor and wavelet analysis,
 the available proofs rely on advanced theory of von Neumann algebras, and may only be accessible to a smaller community of experts. This expository article provides an elementary and self-contained presentation of the density theorem, that is based solely on basic tools from harmonic analysis, representation theory, and frame theory, and should be accessible to an interested non-expert. While almost all methods employed exist in some antecedent form in the different specialized literatures, their particular combination here
makes the basic structure underlying the density theorem transparent;
see Section \ref{sec:technical_overview}. The elementary arguments in this article
fall, however, short of deriving the more general version of Theorem \ref{thm:intro_ICC} by Bekka \cite{bekka2004square}. We hope that this article motivates the non-specialist to delve deeper into operator-algebraic methods.
We also expect that the concrete exposition contributes to the study of quantitative aspects of Theorem \ref{thm:intro_ICC}, such as the relation
between the distance between $\vol(G / \Gamma) d_{\pi}$ and the critical value $1$, and special qualities of the corresponding cyclic or separating vectors,
such as smoothness in the case of Lie groups.

\subsection{Context and related work}
In the setting of Theorem \ref{thm:intro_ICC}, for any non-zero $g \in \Hpi$, the system $\{ \pi (x) g  :  x \in G \}$
is \emph{overcomplete},
i.e., it contains proper subsystems that are still complete.
The fundamental question as to whether subsystems corresponding to lattices \eqref{eq:orbit_intro}
remain complete was posed by Perelomov
in his group-theoretical approach towards the construction
of coherent states \cite{perelomov1972coherent, perelomov1986generalized}. In fact, a criterion for the completeness of subsystems of coherent states similar to Theorem \ref{thm:intro_ICC}
was posed as a question in \cite[p.226]{perelomov1972coherent} \footnote{
Perelomov uses the term coherent state with a slightly different meaning,
as systems are not the full orbit of a group representation, but
parametrized by a homogeneous space to eliminate redundancies. See
\cite{moscovivi1977coherent,moscovici1978coherent} for the relation between the two notions.}.
These criteria have been considered for specific systems and vectors in, e.g., \cite{bargmann1971on, perelomov1971remark, perelomov1973coherent,
klauder1994wavelets, neretin2006perelomov, ramanathan1995incompletness, monastyrsky1974coherent, groechenig2016completeness}.

The related question as to whether a system \eqref{eq:orbit_intro} is a (discrete) frame is at the core of modern frame theory \cite{daubechies1986painless}
and has, in particular, a long history in Gabor theory \cite{heil2007history}.
The existence of a frame vector is also studied in representation theory,
in whose jargon such a vector is called \emph{admissible} \cite{fuehr2005abstract, fuehr2002admissible}.
While the mere existence of a frame or Riesz vector for a given lattice is quite different from the validity of these properties for one specific vector, there is an interesting interplay between the two problems. In Section \ref{sec:examples} we discuss a selection of examples, including one where Theorem \ref{thm:intro_ICC} yields seemingly unnoticed consequences.

\subsection{Projective versions}
The density theorem can also be formulated for \emph{projective} unitary representations
\cite{bekka2004square, radulescu1998berezin, han2017note, gabardo2003frame},
and allows for applications to representations that are square-integrable only modulo a central subgroup
(as in the case of nilpotent or reductive Lie groups).
The proofs that we present work transparently for projective representations
and we formulate the main results in that generality in Theorem \ref{th_main}.
In the projective setting, the lattice is not assumed to be ICC,
but is assumed to satisfy the weaker \emph{Kleppner condition} \cite{kleppner1962structure}, a compatibility condition between the lattice and the cocycle of the projective representation. The projective formulation greatly simplifies the treatment of concrete examples such as weighted Bergman spaces and Gabor systems in Section \ref{sec:examples}.

\subsection{Technical comments} \label{sec:technical_overview}
The common approach to the density theorem is through the
coupling theory of von Neumann algebras, and a self-contained
presentation in this spirit can be found in \cite{goodman1989coxeter,bekka2000square}. Although we make no explicit reference to the coupling theory, some of the arguments we give are simplifications of standard results, as we point out throughout the text. Most significantly, we circumvent certain technicalities associated with the so-called trace of a group von Neumann algebra.
In finding elementary arguments, we benefited particularly from reading \cite{cowling1991irreducibility, alaoglu1940ergodic, kuhn1992restrictions, fuehr2005abstract, rieffel2004integrable}.

An important simplification in the proof of Theorem \ref{thm:intro_ICC} occurs in the derivation of the necessity of the volume or density conditions for cyclicity and separateness, which also play an essential role in deriving the existence of frame and Riesz vectors. Our argument is inspired by Janssen's ``classroom proof''
of the density theorem for Gabor frames \cite{janssen_density}, and underscores the power of
frame-theoretic methods. In this article such argument is pushed further to yield consequences for cyclicity and separateness. While the necessity of the density conditions for frames and Riesz sequences is an active field of research \cite{balan2006density, fuehr2017density, mitkovski1}, most abstract results are not applicable to groups of non-polynomial growth. It is therefore remarkable that the
particular lattice structure of the systems in question \eqref{eq:orbit_intro} leads to simple and conclusive results.

\section{Preliminaries} \label{sec:prelim}
Throughout the article, the locally compact group $G$ is
assumed to be second countable and unimodular. We fix a Haar measure $\mu_G$ on $G$. Some of the notions below depend on this normalization, but the main results do not.

\subsection{Cocycles and projective representations}
A \emph{cocycle} or \emph{multiplier} on $G$ is a Borel measurable function
$\sigma : G \times G \to \mathbb{T}$ such that
\begin{enumerate}[(i)]
\item For all $x,y,z \in G$, $\sigma(x, yz) \sigma(y,z) = \sigma(xy,z) \sigma(x,y);$
\item For the identity $e \in G$ and all $x \in G$, $\sigma(x,e) = \sigma(e,x) = 1$.
\end{enumerate}
A \emph{projective unitary representation} $(\pi, \Hpi)$
of $G$ on a Hilbert space $\Hpi$ is a mapping $\pi : G \to \mathcal{U} (\Hpi)$
satisfying the following conditions:
\begin{enumerate}[(i)]
\item[(i)] The map $x \mapsto \pi(x)$ is weakly measurable, i.e.,
the map $G \ni x \mapsto \langle \pi(x) f, g \rangle \in \mathbb{C}$ is
Borel for all $f, g \in \Hpi$;
\item[(ii)] There exists a function $\sigma : G \times G \to \mathbb{T}$ such that
$\pi(x) \pi(y) = \sigma(x,y) \pi (xy)$ for all $x, y \in G$;
\item[(iii)] $\pi (e) = I$.
\end{enumerate}
In this case, the map $\sigma$ in (ii) is uniquely determined and it is a cocycle.
A projective unitary representation with cocycle $\sigma$ is called a
\emph{$\sigma$-representation}.

Common examples of a representation space $\Hpi$ are Hilbert spaces of real-variable or complex-variable functions; see Section \ref{sec:examples} for a detailed discussion of some examples.

Given two $\sigma$-representations $(\pi_1, \mathcal{H}_{\pi_1})$ and
$(\pi_2, \mathcal{H}_{\pi_2})$, a linear operator $T : \mathcal{H}_{\pi_1} \to \mathcal{H}_{\pi_2}$ is said to \emph{intertwine} $\pi_1$ and $\pi_2$ if
\[ T \pi_1 (x) = \pi_2 (x) T,
\qquad \mbox{for all }x \in G.
 \]
If a bounded linear operator $T : \mathcal{H}_{\pi_1} \to \mathcal{H}_{\pi_2}$ intertwines $\pi_1$ and $\pi_2$, then $T^* : \mathcal{H}_{\pi_2} \to \mathcal{H}_{\pi_1}$ interwines $\pi_2$ and $\pi_1$.

See \cite{mackey1958unitary, gaal1973linear, varadarajan1985geometry} for background on cocycles and projective representations.

\subsection{Square-integrable $\sigma$-representations}
Let $(\pi, \Hpi)$ be a $\sigma$-representation of $G$.
For $f, g \in \Hpi$, the associated \emph{matrix coefficient}
is defined by
$
C_{g} f (x) = \langle f, \pi(x) g \rangle$ for $x \in G.
$
The $\sigma$-representation $(\pi, \Hpi)$ is
called \emph{square-integrable} if there exists a norm dense subspace $\mathcal{D} \subset \Hpi$ such that
\begin{align} \label{eq:squareintegrable_dense}
C_g f = \langle f, \pi (\cdot) g \rangle \in L^2 (G),
\quad f \in \Hpi, \; g \in \mathcal{D}.
\end{align}
The $\sigma$-representation $(\lambda_G^{\sigma}, L^2 (G))$ given by
\begin{align*}
(\lambda_G^{\sigma} (y) F)(x) = \sigma(y, y^{-1} x) F(y^{-1} x), \qquad F \in L^2 (G),\,
x,y \in G,
\end{align*}
is called the \emph{$\sigma$-regular representation} and satisfies the \emph{covariance property} or \emph{intertwining property}:
\begin{align} \label{eq:covariance}
C_g (\pi(y) f)(x) = \sigma(y, y^{-1} x) C_g f (y^{-1} x) = \big( \lambda_G^{\sigma} (y) C_g f\big) (x),
\quad x,y \in G,
\end{align}
 for all  $f \in \Hpi$, $g \in \mathcal{D}$.

A $\sigma$-representation $(\pi, \Hpi)$ is called \emph{irreducible} if
the only closed $\pi(G)$-invariant subspaces of $\Hpi$ are $\{0\}$ and $\Hpi$
and  is said to be a \emph{discrete series $\sigma$-representation} if it is both
square-integrable and irreducible.

Given a discrete series $\sigma$-representation $(\pi, \Hpi)$,
there exists a unique number $d_{\pi} > 0$, called the \emph{formal dimension} of $\pi$, such that the \emph{orthogonality relations}
\begin{align} \label{eq:ortho_rel}
\int_G \langle \pi (x) f_1, g_1 \rangle \overline{\langle \pi(x) f_2, g_2 \rangle} \; d\mu_G (x)
= d_{\pi}^{-1} \langle f_1, f_2 \rangle \overline{\langle g_1, g_2 \rangle}
\end{align}
hold for all $f_1, f_2, g_1, g_2 \in \Hpi$.

The formal dimension $d_{\pi} > 0$ depends on the choice of Haar measure on $G$, and in certain concrete settings, such as real Lie groups, it can be explicitly computed. The book \cite{corwin1990representations} treats nilpotent Lie groups while \cite{neeb2000holomorphy, knapp1986representation} treats
semisimple Lie groups. Explicit expressions of $d_{\pi}$ for the simplest examples of such groups are also provided in Section \ref{sec:examples}.

See \cite{radulescu1998berezin, robert1983introduction} and \cite[Appendix VII]{neeb2000holomorphy} for more on square-integrable representations.

\subsection{Fundamental domains and lattices}
Let $\Gamma \subseteq G$
be a discrete subgroup.
A left (resp. right) fundamental domain of $\Gamma$ in $G$ is a Borel set $\Omega \subseteq G$
satisfying $G = \Gamma \cdot \Omega$
and $\gamma \Omega \cap \gamma' \Omega  = \emptyset$
(resp. $G = \Omega \cdot \Gamma$ and $\Omega \gamma \cap \Omega \gamma' = \emptyset$)
 for all $\gamma, \gamma' \in \Gamma$ with $\gamma \neq \gamma'$.
 If $\Omega$ is a left (resp. right) fundamental domain, then $\Omega^{-1}$ is a right (resp. left)
 fundamental domain.
 The discrete subgroup $\Gamma \subseteq G$ is called a \emph{lattice}
if it admits a left (or right) fundamental domain of finite measure. Equivalently, a discrete subgroup
$\Gamma$ is a lattice if and only if the quotient $G / \Gamma$ admits a finite $G$-invariant regular Borel measure.
Any two fundamental domains have the same measure, and thus, we may define the \emph{co-volume} of $\Gamma$ as $\vol(G/\Gamma) := \mu_G (\Omega)$.
This depends of course on the choice of the Haar measure for $G$.

Standard examples of lattices are $\mathbb{Z}^d \subseteq \mathbb{R}^d$ and $ \mathrm{SL}(2, \mathbb{Z}) \subseteq \mathrm{SL}(2, \mathbb{R})$. The lattice $\mathbb{Z}^d$ is co-compact in $\mathbb{R}^d$, i.e., $\mathbb{R}^d/\mathbb{Z}^d$ is compact, while $ \mathrm{SL}(2, \mathbb{Z})$ is not co-compact in $\mathrm{SL}(2, \mathbb{R})$.

See \cite{raghunathan1972discrete} and \cite[Appendix B]{bekka2008kazhdan} for more on lattices and fundamental domains.

\subsection{ICC groups and Kleppner's condition} \label{sec:kleppner}
Let $\Gamma$ be a discrete countable group and let $\sigma : \Gamma \times \Gamma \to \mathbb{T}$ be a cocycle.
An element $\gamma_0 \in \Gamma$ satisfying
$\sigma(\gamma_0, \gamma) = \sigma(\gamma, \gamma_0)$ for all elements $\gamma \in \Gamma$
commuting with $\gamma_0$ is called \emph{$\sigma$-regular}.
The pair $(\Gamma, \sigma)$ is said to satisfy \emph{Kleppner's condition} if
the conjugacy class $C_{\gamma_0} := \{ \gamma \gamma_0 \gamma^{-1} \; | \; \gamma \in \Gamma\}$ of any $\sigma$-regular element $\gamma_0 \in \Gamma \setminus \{e\}$ is infinite.
The group $\Gamma$ is called an \emph{infinite conjugacy class} (ICC) group if any conjugacy class $C_{\gamma_0}$ for $\gamma_0 \in \Gamma \setminus \{e\}$ is infinite. Any ICC group $\Gamma$ satisfies Kleppner's condition for any cocycle $\sigma : G \times G \to \mathbb{T}$.

\subsection{Von Neumann algebras}
\label{sec_vn}
Let $\mathcal{H}$ be a separable complex Hilbert space. A net $(T_{\alpha} )_{\alpha \in \Lambda}$ of bounded linear operators $T_\alpha \in \mathcal{B}(\mathcal{H})$ converges in the strong operator topology (SOT) to an operator $T \in \mathcal{B}(\mathcal{H})$ if $T_\alpha f \longrightarrow Tf$ in the norm of $\mathcal{H}$ for all $f \in \mathcal{H}$, and it converges in the weak operator topology (WOT) if $\ip{T_\alpha f}{g} \longrightarrow \ip{Tf}{g}$ for all $f,g \in \mathcal{H}$.

A subalgebra $\mathcal{A} \subseteq \mathcal{B}(\mathcal{H})$ is called a \emph{von Neumann algebra}
if $\mathcal{A}$ is self-adjoint, i.e., $\mathcal{A}=\mathcal{A}^*$, contains the identity $I$ and is weakly closed in $\mathcal{B}(\mathcal{H})$.
The \emph{commutant} $M'$ of a set $M \subseteq \mathcal{B}(\mathcal{H})$ is the class of all bounded linear operators that commute with each operator of $M$, i.e.,
\[
M' := \{ T \in \mathcal{B}(\mathcal{H}) \; : \; TS = ST, \; \forall S \in M \}.
\]
By von Neumann's density theorem (see, e.g., \cite[I.3.4, Corollary 1]{dixmier1981vonneumann}), it follows that if $\mathcal{A} \subseteq \mathcal{B}(\mathcal{H})$ is a self-adjoint algebra containing the identity, then $\mathcal{A}'' := (\mathcal{A}')'$ is contained
in the strong closure of $\mathcal{A}$ in $\mathcal{B}(\mathcal{H})$.
In particular, the double commutant $\mathcal{A}''$ is the smallest von Neumann algebra containing $\mathcal{A}$ and equals the strong and weak closure of $\mathcal{A}$.
Thus, for every operator $T \in \mathcal{A}''$, there exist a net of operators of $\mathcal{A}$ converging to $T$ in the SOT topology. Moreover, by Kaplansky's density theorem (see, e.g., \cite[I.3.5, Theorem 3]{dixmier1981vonneumann}), the net may be assumed to be uniformly bounded in operator norm.

For a family of operators $\mathcal{A} \subset \mathcal{H}$ and a vector $g \in \mathcal{H}$,
 the closed linear span of $\mathcal{A} g = \{ A g \; : \; A \in \mathcal{A}\}$ in $\mathcal{H}$
 is denoted by $[\mathcal{A}g] := \overline{\Span \mathcal{A} g}$.

Given a von Neumann algebra $\mathcal{A} \subseteq \mathcal{B}(\mathcal{H})$ and
an orthogonal projection $P_{\mathcal{K}}$ onto a closed subspace $\mathcal{K} \subseteq \mathcal{H}$, the space $\mathcal{K}$ is invariant under $\mathcal{A}$, i.e., $\mathcal{A}(\mathcal{K}) \subset \mathcal{K}$, if and only if $P_{\mathcal{K}} \in \mathcal{A}'$.
This observation is known as the \emph{projection lemma}.
For more background on von Neumann algebras, see \cite{gaal1973linear, kadison1983fundamentals1, dixmier1981vonneumann}.

\subsection{Partial isometries and the polar decomposition}
Let $\mathcal{H}$ and $\mathcal{K}$ be complex Hilbert spaces.
A bounded linear operator $U : \mathcal{K} \to \mathcal{H}$ is called a \emph{partial isometry}
if $U$ is an isometry when restricted to the orthogonal complement $\mathcal{N}(U)^{\perp}$
of its null space $\mathcal{N}(U)$.  The subspace $\mathcal{N}(U)^{\perp}$ is called the \emph{initial space} of $U$ and the range $\mathcal{R} (U)$ of $U$ is the \emph{final space} of $U$, i.e., the image of $\mathcal{N}(U)^{\perp}$ under the isometry $U|_{\mathcal{N} (U)^{\perp}}$

A linear operator $T : \dom(T) \subset \mathcal{H} \to \mathcal{K}$
is \emph{densely defined} if its domain $\dom(T)$ is a norm dense subspace in $\mathcal{H}$
and is called \emph{closed} if its graph $\mathcal{G}(T) := \{ (f, Tf) \; | \; f \in \mathcal{H} \}$
is closed in $\mathcal{H} \oplus \mathcal{K}$.
For a closed, densely defined linear operator $T : \dom(T) \subset \mathcal{H} \to \mathcal{K}$,
its \emph{adjoint} is denoted by $T^*$ and its \emph{modulus} by $|T| := (T^* T)^{1/2}$.
The operator $|T|$ is defined by Borel functional calculus and has domain $\dom(|T|)=\dom(T)$.
The \emph{polar decomposition} of $T$ is uniquely given by
\[
T = U_T |T| = |T^*| U_T,
\]
where $U_T : \mathcal{H} \to \mathcal{K}$ is a partial isometry
with initial space $\mathcal{N}(T)^{\perp} = \overline{\mathcal{R}(|T|)}$
and final space $\overline{\mathcal{R}(T)}$. For more details and background, see, e.g., \cite[VI, Section 13]{fell1988representations1}.

\section{Orbits of square-integrable representations} \label{sec:vectors_frames}
Let $(\pi, \Hpi)$ be a square-integrable $\sigma$-representation of a countable discrete group $\Gamma$ on a separable (complex) Hilbert space $\Hpi$.
For a vector $g \in \Hpi$, we consider the orbit $\pi(\Gamma) g$ of $g$ under $(\pi, \Hpi)$, i.e.,
\[ \CS := \big\{ \pi (\gamma) g : \gamma \in \Gamma \big\}. \]
We treat the system $\CS$ as a family indexed by $\Gamma$ and
allow for repetitions.

\subsection{Cyclic and separating vectors}
\label{sec_comp_sep}

A vector $g \in \Hpi$ is called \emph{cyclic} or \emph{complete}  if $[\pi(\Gamma) g] = \Hpi$.
By von Neumann's density theorem, the vector $g \in \Hpi$ is cyclic if and only if $[\pi (\Gamma)'' g ] = \Hpi$. A vector $g \in \Hpi$ is called \emph{separating} for $\pi (\Gamma)''$
 if $T \in \pi (\Gamma)''$ and $Tg = 0$ imply $T = 0$, that is, if the map $\pi(\Gamma)'' \ni T \mapsto Tg \in \Hpi$ is injective.

A vector $g \in \Hpi$ is separating for $\pi(\Gamma)''$ if and only if $[\pi(\Gamma)' g]= \Hpi$.
Indeed, if $[\pi (\Gamma)'g] \neq \Hpi$, then the projection $P_{\mathcal{K}}$ onto $\mathcal{K} := [\pi(\Gamma)' g]$ is in $\pi(\Gamma)''$ and $P_{\mathcal{K}} \neq I$. Thus $I-P_{\mathcal{K}} \neq 0$ and $(I - P_{\mathcal{K}} )g = 0$, showing that $g$ is not separating for $\pi(\Gamma)''$.
Conversely, if $[\pi(\Gamma)' g] = \Hpi$ and $T \in \pi(\Gamma)''$ is such that $T g = 0$, then
$0 = S Tg = T S g$ for all $S \in \pi(\Gamma)'$, and hence $T = 0$ since $\pi(\Gamma)' g$ is norm dense in $\Hpi$.

Intuitively, a vector $g \in \Hpi$ is cyclic if the corresponding orbit $\pi(\Gamma) g$ is rich enough so as to provide approximations for every vector in $\Hpi$.
On the other hand, if $g$ is separating for $\pi(\Gamma)''$,
then $\pi (\Gamma)''$ cannot be too rich, because $\pi(\Gamma)'' \ni T \mapsto Tg \in \Hpi$ is injective.

The central question of this article is the relation between the existence of cyclic and separating vectors on the one hand, and the co-volume of $\Gamma$ within a larger group $G$. As a key tool, we consider certain strengthened notions of cyclicity and separation.

\subsection{Frames and Riesz sequences}
A system $\CS$ is called a \emph{frame} for $\Hpi$ if there exist constants $A, B > 0$,
called \emph{frame bounds}, such that the following \emph{frame inequalities} hold:
\begin{align}
\label{eq_frame_bounds}
A \| f \|_{\Hpi}^2 \leq \sum_{\gamma \in \Gamma} |\langle f, \pi (\gamma) g \rangle |^2
\leq B \| f \|_{\Hpi}^2, \qquad f \in \Hpi.
\end{align}
A vector $g$ is a \emph{frame vector} if $\pi(\Gamma) g$ is a frame.
A system $\CS$ forming a frame is complete by the first (lower) bound in \eqref{eq_frame_bounds}. The second of the frame inequalities (upper bound),
\begin{align}
\label{eq_Bessel}
\sum_{\gamma  \in \Gamma} |\langle f, \pi (\gamma) g \rangle |^2 \leq B \| f \|_{\Hpi}^2
, \qquad f \in \Hpi,
\end{align}
is known as a \emph{Bessel bound}. A vector $g$ satisfying \eqref{eq_Bessel}
is a \emph{Bessel vector}. Note that the definition concerns $\pi(\Gamma) g$ as an indexed family. Two indexations of the same underlying set can have, for example, different frame bounds.
The frame bounds of a given frame and indexation are of course not unique.

The Bessel condition \eqref{eq_Bessel} is equivalent to the
 \emph{frame operator}
\[ S_{g, \Gamma} : \Hpi \to \Hpi, \;
S_{g, \Gamma} f = \sum_{\gamma \in \Gamma} \langle f, \pi(\gamma) g \rangle \pi(\gamma) g
 \]
being well-defined and bounded. The full two-sided frame inequality \eqref{eq_frame_bounds}
is equivalent to the frame operator being a positive-definite (bounded, invertible) operator
on $\Hpi$. A frame $\CS$ for which the frame bounds can be chosen
as $A=B=1$ is called a \emph{Parseval frame}, because it gives the identity
\begin{align*}
\| f \|_{\Hpi}^2 =
\sum_{\gamma  \in \Gamma} |\langle f, \pi (\gamma) g \rangle |^2, \qquad f \in \Hpi.
\end{align*}
Equivalently,
$\CS$ is a Parseval frame for $\Hpi$ if and only if its frame operator $S_{g,\Gamma}$ is the identity on $\Hpi$.
Whenever well-defined and bounded, the frame operator $S_{g, \Gamma}$ commutes with
$\pi(\gamma)$ for all $\gamma \in \Gamma$.

\begin{remark}[Turning a frame into a Parseval one]
\label{rem_parsevalization}
An arbitrary frame $\CS$ can be turned into a Parseval  frame by considering $\tilde g := S_{g, \Gamma}^{-1/2} g$. Indeed, if $\CS$ is a frame, then
$S_{g, \Gamma}$ is a positive operator, and, therefore, $\tilde g$ is well-defined. Moreover, since $S_{g, \Gamma}^{-1/2}$ also commutes with each $\pi(\gamma)$, for $f \in \Hpi$,
\begin{align*}
S_{\tilde g, \Gamma} f =
\sum_{\gamma \in \Gamma} \langle f, \pi(\gamma) S_{g, \Gamma}^{-1/2} g \rangle \pi(\gamma) S_{g, \Gamma}^{-1/2}g
= S_{g, \Gamma}^{-1/2} S_{g, \Gamma} S_{g, \Gamma}^{-1/2} f
= f,
\end{align*}
showing that $\pi(\Gamma) \tilde{g}$ is a Parseval frame for $\Hpi$.
\end{remark}

A system $\CS$ is called a \emph{Riesz sequence} in $\Hpi$ if there exist
constants $A, B > 0$, called \emph{Riesz bounds}, such that
\[
A \| c  \|_{\ell^2}^2 \leq \bigg\| \sum_{\gamma \in \Gamma} c_{\gamma} \pi (\gamma) g \bigg\|_{\Hpi}^2 \leq B \| c \|_{\ell^2}^2, \quad c = (c_{\gamma})_{\gamma \in \Gamma} \in \ell^2 (\Gamma).
\]
A duality argument, shows that a Riesz sequence satisfies the Bessel bound \eqref{eq_Bessel}. Moreover, a Riesz sequence is linearly independent and $\omega$-independent, and hence cannot admit repetitions.
A vector $g$ yielding a Riesz sequence $\pi(\Gamma) g$ is a \emph{Riesz vector}.

A complete Riesz sequence $\CS$ is called a \emph{Riesz basis} for $\Hpi$.
Equivalently, a system $\CS$ is a Riesz basis for $\Hpi$ if it is the image of an orthonormal basis under
a bounded, invertible operator on $\Hpi$.
If $\pi(\Gamma) g$ is a Riesz basis for $\Hpi$,
then $\pi(\Gamma) g$ and $ \pi(\Gamma) S^{-1}_{g,\Gamma} g = S^{-1}_{g,\Gamma} \pi(\Gamma) g$ are biorthogonal sequences in $\Hpi$, i.e., $\langle \pi(\gamma') g, S_{g,\Gamma}^{-1} \pi(\gamma) g \rangle = \delta_{\gamma', \gamma}$ for $\gamma, \gamma' \in \Gamma. $

It will be shown in Proposition \ref{prop_riesz_sep} that, under Kleppner's condition, if $\pi(\Gamma)g$ is a Riesz sequence, then $g$ is separating for $\pi(\Gamma)''$.

\begin{remark}[Turning a Riesz sequence into an orthonormal one]
\label{rem_Riesz_on}
If $\pi(\Gamma)g$ is a Riesz sequence in $\Hpi$,
then it is a Riesz basis for $[\pi(\Gamma) g]$ and hence
the frame operator $S_{g, \Gamma} : [\pi(\Gamma) g ] \to [\pi(\Gamma) g]$
is well-defined and bounded. The biortogonality of the systems
$\pi(\Gamma)g$ and $\pi (\Gamma) S^{-1}_{g,\Gamma} g$ yields that
\[
\big\langle S^{-1/2}_{g, \Gamma} \pi(\gamma') g, S^{-1/2}_{g, \Gamma} \pi(\gamma) g \big\rangle
= \big \langle \pi(\gamma') g, S^{-1}_{g, \Gamma} \pi(\gamma) g \big\rangle = \delta_{\gamma', \gamma},
\quad \gamma, \gamma' \in \Gamma,
\]
showing that $\pi(\Gamma) S^{-1/2}_{g,\Gamma} g = S^{-1/2}_{g, \Gamma} \pi(\Gamma) g$
is an orthonormal sequence in $\Hpi$.
\end{remark}

For more on frames and Riesz bases, see, e.g., the books \cite{christensen2016introduction, young2001introduction}.

\subsection{Bounded operators and Bessel vectors}
The \emph{coefficient operator} and \emph{reconstruction operator} associated with $\CS$
are given respectively by
\begin{align} \label{eq:coefficient}
C_{g,\Gamma} f =  \big( \langle f, \pi(\gamma) g \rangle \big)_{\gamma \in \Gamma},
\quad f \in \Hpi,
\end{align}
and
\begin{align} \label{eq:reconstruction}
D_{g, \Gamma} c =  \sum_{\gamma \in \Gamma} c_{\gamma} \pi(\gamma) g,
\quad c = (c_{\gamma})_{\gamma \in \Gamma} \in c_{00} (\Gamma),
\end{align}
where $c_{00} (\Gamma) \subseteq \mathbb{C}^{\Gamma}$
denotes the space of finite sequences on $\Gamma$.

Recall that $\CS$ is called a \emph{Bessel sequence} if there exists  $B > 0$ such that
\eqref{eq_Bessel} holds. In this case, the coefficient operator is well-defined and bounded
as a map from $\Hpi$ into $\ell^2 (\Gamma)$, and its adjoint $D_{g, \Gamma}$ is well-defined and bounded from $\ell^2 (\Gamma)$ into $\Hpi$.

The space of Bessel vectors is denoted by $\mathcal{B}_{\pi}$.
The assumption that $(\pi, \Hpi)$ is square-integrable in the sense of \eqref{eq:squareintegrable_dense},
together with the uniform boundedness principle,
yields that the space $\Bpi$ is norm dense in $\Hpi$.

\subsection{Coefficient and reconstruction as unbounded operators}
In the sequel, we treat the coefficient mapping \eqref{eq:coefficient}
and reconstruction mapping \eqref{eq:reconstruction} as operators from domains
and on images in which they do not necessarily act as bounded operators.

The coefficient operator $C_{g, \Gamma}$,
with domain
\[
\dom(C_{g, \Gamma}) :=  \big\{f \in \mathcal{H}_{\pi} \; : \; C_{g, \Gamma} f \in \ell^2 (\Gamma) \big\}
\]
is given by $f \mapsto ( \langle f, \pi(\gamma) g \rangle )_{\gamma \in \Gamma}$
and well-defined from $\dom(C_{g,\Gamma})$ into $\ell^2 (\Gamma)$.

The reconstruction operator $D_{g, \Gamma}$, with domain
\begin{align} \label{eq:dom_reconstruction}
\dom(D_{g, \Gamma} ) :=
\bigg\{ c \in \ell^2 (\Gamma) \; \bigg  | \; \exists f \in \mathcal{H}_{\pi} \; : \; \sum_{\gamma \in \Gamma} c_{\gamma} \langle \pi (\gamma) g, h \rangle \
= \langle f, h \rangle, \; \forall h \in \mathcal{B}_{\pi} \bigg\}
\end{align}
is given by $D_{g, \Gamma} c = f$ and well-defined from $\dom(D_{g, \Gamma})$ into $\Hpi$,
where $f \in \Hpi$ is the vector occurring in the domain definition \eqref{eq:dom_reconstruction}.
Note that $f$ is uniquely determined since $\Bpi$ is a dense subspace in $\Hpi$.

For simplicity, we also sometimes write
\begin{align*}
D_{g, \Gamma} c = \sum_{\gamma \in \Gamma} c_\gamma \pi(\gamma) g;
\end{align*}
the series is however a formal expression for the vector $f$ in \eqref{eq:dom_reconstruction}.

The following result provides basic properties of the (possibly) unbounded coefficient and reconstruction operators.

\begin{proposition} \label{prop:denselydefined}
Let $(\pi, \Hpi)$ be a square-integrable $\sigma$-representation of a countable discrete group $\Gamma$.
Let $g \in \Hpi$ be an arbitrary vector.
\begin{enumerate}[(i)]
\item
The coefficient operator $C_{g, \Gamma} : \dom(C_{g, \Gamma}) \to \ell^2 (\Gamma)$, $C_{g, \Gamma} f = ( \langle f,  \pi(\gamma) g \rangle )_{\gamma \in \Gamma}$
is a closed, densely defined operator.
\item The reconstruction operator $D_{g, \Gamma}  : \dom(D_{g, \Gamma}) \to \mathcal{H}_{\pi}$, $D_{g, \Gamma} c = f$,
is a closed, densely defined operator.
\end{enumerate}
\end{proposition}
\begin{proof}
(i) The map  $C_{g, \Gamma} : \dom(C_{g, \Gamma}) \to \ell^2 (\Gamma)$ is densely defined
since the dense space of Bessel vector $\Bpi \subset \dom(C_{g, \Gamma})$.
To show that $C_{g, \Gamma}$ is closed, let $f_n \to f$ in $\Hpi$ with $f_n \in \dom(C_{g, \Gamma})$
and assume that $C_{g, \Gamma} f_n \to c$ in $\ell^2 (\Gamma)$ as $n \to \infty$.
By Cauchy-Schwarz,
\[
|C_{g, \Gamma} f_n (\gamma) - C_{g, \Gamma} f (\gamma)| = |\langle f_n - f, \pi(\gamma)g \rangle | \leq \| f_n - f \|_{\Hpi} \| g \|_{\Hpi} \to 0
\]
as $n \to \infty$, yielding that $c = C_{g, \Gamma} f$. This shows that $C_{g,\Gamma} f \in \ell^2 (\Gamma)$, and hence $f \in \dom(C_{g,\Gamma})$.

(ii) Note that the map $D_{g, \Gamma}$ is densely defined
since the space of finite sequences $c_{00} (\Gamma) \subseteq \dom(D_{g, \Gamma})$.
To show that $D_{g, \Gamma}$ is closed,
let $(c^{(k)})_{k \in \mathbb{N}} \subset \dom(D_{g, \Gamma})$ be such that
$c^{(k)} \to c$ in $\ell^2 (\Gamma)$ and $f_k := D_{g, \Gamma} c^{(k)} \to f$ for some $f \in \mathcal{H}_{\pi}$
as $k \to \infty$.
Let $h \in \mathcal{B}_{\pi}$ be arbitrary. Then,
\[
\big\langle c^{(k)}, C_{g, \Gamma} h \rangle_{\ell^2 (\Gamma)} = \sum_{\gamma \in \Gamma} c_{\gamma}^{(k)} \langle \pi (\gamma) g, h \rangle = \langle f_k , h \rangle.
\]
Since $C_{g,\Gamma} h \in \ell^2 (\Gamma)$ as $h \in \mathcal{B}_{\pi}$, it follows that
$\langle c^{(k)}, C_{g,\Gamma} h \rangle_{\ell^2 (\Gamma)} \to \langle c, C_{g, \Gamma} h \rangle$ as $k \to \infty$, and hence
\[
 \sum_{\gamma \in \Gamma} c_{\gamma} \langle \pi (\gamma) g, h \rangle
 = \lim_{k \to \infty} \big\langle c^{(k)}, C_{g,\Gamma} h \big\rangle_{\ell^2 (\Gamma)}
 = \lim_{k \to \infty} \langle f_k , h \rangle
 = \langle f, h \rangle.
\]
Thus $c \in \dom(D_{g, \Gamma})$ and $D_{g, \Gamma} c = f$,
which shows that $D_{g, \Gamma}$ is a closed operator.
\end{proof}

\begin{remark}
For a general frame $\{f_i\}_{i \in I}$ in an abstract Hilbert space $\mathcal{H}$,
the coefficient operator $f \mapsto ( \langle f, f_i \rangle )_{i \in I}$ is always closed,
but not necessarily densely defined, on its canonical domain.
The reconstruction operator $(c_i)_{i \in I} \mapsto \sum_{i \in I} c_i f_i$
may fail to be closed on the domain
\[
\bigg\{ c = (c_i)_{i \in I} \in \ell^2 (I) \; : \; \sum_{i \in I} c_i f_i
\mbox{ converges in the norm of }\mathcal{H} \bigg\},
\]
see \cite{christensen1995frames}.
Crucially, in \eqref{eq:dom_reconstruction} and part (ii) of Proposition \ref{prop:denselydefined}, we define the series in a suitably weak form.
\end{remark}

\subsection{Uniqueness for the extended representation} \label{sec:uniqueness}
Given  $c = (c_{\gamma})_{\gamma \in \Gamma} \in \ell^2 (\Gamma)$,
define the operator
\begin{align}
\label{eq_pic}
\pi(c) : \Bpi \to \Hpi, \quad  \pi(c) g :=\sum_{\gamma \in \Gamma} c_\gamma \pi(\gamma) g.
\end{align}
Note that $\pi(c) = \sum_{\gamma \in \Gamma} c_{\gamma} \pi(\gamma)$ is well-defined since the series representing $\pi(c)g$ converges
unconditionally in $\Hpi$ by the Bessel property.

In the notation of \eqref{eq_pic}, conjugating the operator $\pi(c)$ simply
corresponds to (twisted) conjugation of the corresponding sequence $c$.

\begin{lemma} \label{lem:conjugate}
Let $(\pi, \Hpi)$ be a square-integrable $\sigma$-representation
of a countable discrete group $\Gamma$.
Let $c  \in \ell^2 (\Gamma)$.
Then, for all $\gamma \in \Gamma$,
\[
\pi(\gamma) \pi (c) \pi(\gamma)^*
= \pi(\vartheta_{\Gamma}^{\sigma} (\gamma) c),
\]
where
\begin{align} \label{eq:twosided_translation}
(\vartheta_\Gamma^{\sigma} (\gamma) c)_{\gamma'} := \overline{\sigma(\gamma^{-1}, \gamma')} \sigma(\gamma^{-1} \gamma' \gamma, \gamma^{-1}) c_{\gamma^{-1} \gamma' \gamma}, \qquad \gamma, \gamma' \in \Gamma.
\end{align}
\end{lemma}
\begin{proof}
Let $\gamma \in \Gamma$ be fixed.
The identity
$
\pi(\gamma) \pi( \gamma') \pi(\gamma)^* = \sigma(\gamma, \gamma') \overline{\sigma(\gamma \gamma' \gamma^{-1}, \gamma)} \pi(\gamma \gamma' \gamma^{-1})
$
holds for any $\gamma' \in \Gamma$.
Therefore,
\begin{align*}
\pi(\gamma) \pi (c) \pi(\gamma)^* &= \sum_{\gamma' \in \Gamma} c_{\gamma'}
\sigma(\gamma, \gamma') \overline{\sigma(\gamma \gamma' \gamma^{-1}, \gamma)} \pi(\gamma \gamma' \gamma^{-1}) \\
&= \sum_{\gamma' \in \Gamma} c_{\gamma^{-1} \gamma' \gamma} \sigma(\gamma, \gamma^{-1} \gamma' \gamma) \overline{\sigma(\gamma', \gamma)} \pi(\gamma'), \numberthis \label{eq:conjugate1}
\end{align*}
where the second equality follows from the change of variable $\gamma' \mapsto \gamma \gamma' \gamma^{-1}$. Combining the identity
\[
\sigma(\gamma, \gamma^{-1} \gamma' \gamma) \sigma(\gamma' \gamma, \gamma^{-1})
= \sigma(\gamma, \gamma^{-1} \gamma') \sigma(\gamma^{-1} \gamma' \gamma, \gamma^{-1})
\]
with
\[
\sigma(\gamma, \gamma^{-1} \gamma') \sigma(\gamma^{-1}, \gamma')
= \sigma(\gamma, \gamma^{-1}) = \sigma(\gamma', \gamma) \sigma(\gamma' \gamma, \gamma^{-1}),
\]
yields that $\sigma(\gamma, \gamma^{-1} \gamma' \gamma) \overline{\sigma(\gamma', \gamma)} = \overline{\sigma(\gamma^{-1}, \gamma')} \sigma(\gamma^{-1} \gamma' \gamma, \gamma^{-1})$
for all $\gamma' \in \Gamma$. Inserting this in \eqref{eq:conjugate1} gives
\begin{align*}
\pi(\gamma) \pi (c) \pi(\gamma)^* &= \sum_{\gamma' \in \Gamma} c_{\gamma^{-1} \gamma' \gamma} \overline{\sigma(\gamma^{-1}, \gamma')} \sigma(\gamma^{-1} \gamma' \gamma, \gamma^{-1}) \pi(\gamma')
= \sum_{\gamma' \in \Gamma} \big(\vartheta_{\Gamma}^{\sigma} (\gamma) c\big)_{\gamma'} \pi (\gamma'),
\end{align*}
as desired.
\end{proof}

Under Kleppner's condition (see Section \ref{sec:kleppner}), we have the following important uniqueness result.
\begin{proposition}
\label{prop_uniq_op}
Let $(\pi, \Hpi)$ be a square-integrable $\sigma$-representation
of a countable discrete group $\Gamma$. Suppose that $(\Gamma, \sigma)$
satisfies Kleppner's condition. Suppose $c \in \ell^2(\Gamma)$ is such that
$\pi(c) \equiv 0$ on $\mathcal{B}_{\pi}$. Then $c = 0$.
\end{proposition}
\begin{proof}
The proof is divided into four steps.

\textbf{Step 1.} (Invariance of kernel).
Let $\vartheta = \vartheta_{\Gamma}^{\sigma}$ be the unitary action of $\Gamma$ on $\ell^2 (\Gamma)$
given by \eqref{eq:twosided_translation}.
Define the closed subspace
\[
\mathcal{K} := \bigg\{ c \in \ell^2 (\Gamma) \; : \;
\pi(c) g = 0,
\quad \forall g \in B_{\pi} \bigg\}
= \bigcap_{g \in \mathcal{B}_{\pi}} \mathcal{N} (D_{g, \Gamma})
\]
of $\ell^2 (\Gamma)$. The space $\mathcal{K}$ is $\vartheta (\Gamma)$-invariant.
Indeed, for $c \in \mathcal{K}$, by Lemma \ref{lem:conjugate},
\begin{align*}
\bigg\langle \pi\big( \vartheta(\gamma)c\big) g, h \bigg \rangle
= \bigg\langle \sum_{\gamma' \in \Gamma} c_{\gamma'} \pi (\gamma') \pi(\gamma)^* g, \pi(\gamma)^* h \bigg\rangle = 0
\end{align*}
for all $\gamma \in \Gamma$ and $g, h \in \Bpi$.
Moreover, the space $\mathcal{K}$ is $\lambda_{\Gamma}^{\sigma} (\Gamma)$-invariant:
For $\gamma \in \Gamma$ and $g,h \in \mathcal{B}_{\pi}$,
\begin{align*}
\bigg\langle \pi \big(\lambda_{\Gamma}^{\sigma} (\gamma) c \big) g, h \bigg\rangle
&= \sum_{\gamma' \in \Gamma} \sigma(\gamma, \gamma^{-1} \gamma') c_{\gamma^{-1} \gamma'}
\langle \pi(\gamma') g,h \rangle \\
 &= \sum_{\gamma' \in \Gamma} \sigma(\gamma, \gamma')
\overline{\sigma(\gamma, \gamma')} c_{\gamma'} \langle \pi(\gamma) \pi(\gamma') g, h \rangle \\
&= \bigg \langle \sum_{\gamma' \in \Gamma} c_{\gamma'}  \pi(\gamma') g, \pi(\gamma)^* h \bigg\rangle = 0,
\end{align*}
where the second equality follows from the change of variable $\gamma' \mapsto \gamma \gamma'$.

\textbf{Step 2.} (Minimal fixed point).
Let $c \in \mathcal{K}$ be arbitrary and consider the norm-closed convex hull
$\overline{\co(\vartheta (\Gamma) c)}$ in the Hilbert space $\mathcal{K}$.
Then there exists a unique $d \in \overline{\co(\vartheta (\Gamma) c)}$
of minimal norm. By uniqueness, the vector $d$
must be $\vartheta (\Gamma)$-invariant, that is,
\begin{align} \label{eq:d_constant}
d_{\gamma'} = \overline{\sigma(\gamma^{-1}, \gamma')} \sigma(\gamma^{-1} \gamma' \gamma, \gamma^{-1}) d_{\gamma^{-1} \gamma' \gamma},
\qquad \mbox{for all }\gamma, \gamma' \in \Gamma.
\end{align}
Therefore, $|d|$ is constant on conjugacy classes.

\textbf{Step 3.} ($\sigma$-regularity of non-zero entries).
Let $\gamma' \in \Gamma$ be such that $d_{\gamma'} \neq 0$.
Suppose $\gamma \in \Gamma$ commutes with $\gamma'$.
Then, by \eqref{eq:d_constant},
\begin{align} \label{eq:sigma-regular}
0 \neq d_{\gamma'} = \overline{\sigma(\gamma, \gamma')} \sigma(\gamma \gamma' \gamma^{-1}, \gamma ) d_{\gamma \gamma' \gamma^{-1}} =
 \overline{\sigma(\gamma, \gamma')} \sigma(\gamma', \gamma) d_{\gamma'},
\end{align}
and, therefore, $\sigma(\gamma, \gamma') = \sigma(\gamma', \gamma)$.
Thus $\gamma'$ is $\sigma$-regular.

\textbf{Step 4.} (Vanishing coefficients on regular classes).
By Step 3, if $\gamma \in \Gamma$ is such that $d_\gamma \not= 0$, then $\gamma$ is $\sigma$-regular, and by Kleppner's condition,
the conjugacy class $C_\gamma$ is infinite, unless $\gamma=e$. On the other hand, by \eqref{eq:d_constant}, $|d|$ is constant
on $C_\gamma$, while $|d| \in \ell^2 (\Gamma)$, and therefore $C_\gamma$ must be finite. We conclude that
$d_{\gamma} = 0$ for all $\gamma \in \Gamma \setminus \{e\}$.
Moreover, since $d \in \mathcal{K}$, also $d_e = 0$, and hence $d=0$.

\textbf{Step 5.} (Conclusion). The above shows that for an arbitrary $c \in \mathcal{K}$,
we have
$0 \in \overline{\co(\vartheta (\Gamma) c)}$.
Since $(\vartheta(\gamma) c)_e = c_e$ for all $\gamma \in \Gamma$,
it follows that $c_e = 0_e = 0$. The $\lambda^{\sigma}_{\Gamma} (\Gamma)$-invariance
of $\mathcal{K}$ now yields that $c_{\gamma} = \overline{\sigma(\gamma^{-1}, \gamma)} \big(\lambda_{\Gamma}^{\sigma} (\gamma^{-1}) c\big)_e = 0$ for all $\gamma \in \Gamma$.
This completes the proof.
\end{proof}

Step 2 in the proof of Proposition \ref{prop_uniq_op} is an application of the \emph{minimal method} for ergodic theorems \cite[Section 10]{alaoglu1940ergodic}.

For the Heisenberg projective representation $(\pi, L^2 (\mathbb{R}^d))$ of a lattice $\Gamma \leq \mathbb{R}^{2d}$ (see Section \ref{sec:gabor}), an alternative proof for the uniqueness result of Proposition \ref{prop_uniq_op} can be given using the uniqueness of coefficients in Fourier series, see \cite[Proposition 3.2]{groechenig2007gabor}. In that setting, the statement of Proposition \ref{prop_uniq_op} is true even without Kleppner's condition, while in general it is not, e.g., for $\Gamma = \mathrm{SL}(2, \mathbb{Z})$ and a holomorphic discrete series representation $\pi$ of $\mathrm{SL}(2, \mathbb{R})$, cf. Example \ref{sec:SL_ICC}.

\section{Improving spanning properties} \label{sec:improving}
\subsection{Mackey-type version of Schur's lemma}
We will repeatedly use the following folklore result.

\begin{proposition} \label{prop:intertwiner}
For $i \in \{1,2\}$, let $(\pi_i, \mathcal{H}_{\pi_i})$
be  $\sigma$-representations of a locally compact group $G$.
Suppose that $T : \mathcal{H}_{\pi_1} \to \mathcal{H}_{\pi_2}$
is a closed, densely defined operator intertwining $(\pi_1, \mathcal{H}_{\pi_1})$ and $(\pi_2, \mathcal{H}_{\pi_2})$; that is, the domain and range of $T$ are respectively $\pi_1(G)$ and $\pi_2(G)$-invariant, and
\begin{align*}
T \pi_1(x) = \pi_2(x) T, \quad x \in G.
\end{align*}
If
\[
T = U |T|
\]
is the polar decomposition of $T$, then $|T| : \dom (T) \to \mathcal{H}_{\pi_2}$ commutes with $(\pi_1, \mathcal{H}_{\pi_1})$ and the isometry
$U : \mathcal{N}(T)^{\perp} \to \overline{\mathcal{R}(T)}$
isometrically intertwines $(\pi_1, \mathcal{H}_{\pi_1})$ and $(\pi_2, \mathcal{H}_{\pi_2})$.
\end{proposition}
\begin{proof}
Note that $\pi_i (x)^* = \overline{\sigma(x, x^{-1})} \pi_i (x^{-1})$
and let $\tau(x) := \overline{\sigma(x, x^{-1})} \in \mathbb{T}$ for $x \in G$.
Using that $\pi_1 (x)^* T^* = T^* \pi_2 (x)^*$
for all $x \in G$, a direct calculation entails
\[
T^* T \pi_1 (x) = T^* \pi_2 (x) T = T^* \overline{\tau (x)} \pi_2 (x^{-1})^* T
= \overline{\tau(x)} \pi_1 (x^{-1})^* T^* T
= \pi_1 (x) T^* T,
\]
showing that $T^*T$ intertwines $(\pi_1, \mathcal{H}_{\pi_1})$.
The operator $|T|$ is obtained from $T^*T$ by Borel functional calculus, and thus also
commutes with $(\pi_1, \mathcal{H}_{\pi_1})$, e.g., see
 \cite[Theorem 12.14]{fell1988representations1}.
Using this, it follows directly that
\[
U \pi_1 (x) |T| = U |T| \pi_1 (x) = \pi_2 (x) U |T|,
\]
whence $(U \pi_1 (x) - \pi_2 (x) U)|T| = 0$ for $x \in G$.
Hence $(U \pi_1 (x) - \pi_2 (x) U) \equiv 0$ on $\mathcal{R}(|T|)$.
Since  $\mathcal{R}(|T|)$ is dense in $\mathcal{N}(T)^{\perp} = \overline{\mathcal{R}(|T|)}$,
the desired conclusion follows.
\end{proof}

Mackey-type versions of Schur's lemma for representations
of $*$-algebras can be found in \cite{fell1988representations1}.

\subsection{From cyclic vectors to Parseval frames}
We show the existence of Parseval frames $\CS$
whenever $\pi$ admits a complete vector.

\begin{proposition} \label{prop:cyclic_tight-frame}
Let $(\pi, \Hpi)$
be a square-integrable $\sigma$-representation of a countable discrete group $\Gamma$.
Let $h \in \Hpi$ be arbitrary.
Then there exists $g \in \mathcal{H}_{\pi}$ such that $\CS$
is a Parseval frame for $[\pi(\Gamma)h]$.
In particular, if $\pi$ is cyclic, then there exists a Parseval frame $\CS$ for $\Hpi$.
\end{proposition}
\begin{proof} We split the proof into two steps.

\textbf{Step 1.} (\emph{Unitary intertwiner}).
For  $h \in \Hpi$,
the map $C_{h,\Gamma} : \dom(C_{h, \Gamma}) \subseteq \mathcal{H}_{\pi} \to \ell^2 (\Gamma)$ is closed and densely defined by Proposition \ref{prop:denselydefined}.
Moreover, $C_{h,\Gamma}$ intertwines $\pi$ and $\lambda^{\sigma}_{\Gamma}$ by the covariance property \eqref{eq:covariance}.
Thus the partial  isometry $U : \mathcal{N} (C_{h,\Gamma})^{\perp} \to \overline{\mathcal{R} (C_{h,\Gamma})}$ of the decomposition $C_{h, \Gamma} = U |C_{h, \Gamma}|$ intertwines $\pi$ and $\lambda^{\sigma}_{\Gamma}$
by Proposition \ref{prop:intertwiner}.
Since $\mathcal{N} (C_{h,\Gamma})^{\perp} = [\pi(\Gamma) h]$,
it follows that
 $U :[\pi(\Gamma) h] \to \overline{\mathcal{R} (C_{h, \Gamma})}$
 is a unitary intertwiner.

\textbf{Step 2.} (\emph{Parseval frame}).
Let $P_{\mathcal{K}} : \ell^2 (\Gamma) \to \ell^2 (\Gamma)$ be the orthogonal projection
onto $\mathcal{K} := \overline{\mathcal{R} (C_{\eta,\Gamma})}$.
Then $P_{\mathcal{K}} \in \lambda^{\sigma}_{\Gamma} (\Gamma)'$ by the projection lemma,
and $\lambda^{\sigma}_{\Gamma} (\Gamma) P_{\mathcal{K}} \delta_{e}
= P_{\mathcal{K}} \lambda_{\Gamma}^{\sigma} (\Gamma) \delta_e$
satisfies
\[
\| f \|_{\Hpi}^2 = \|P_{\mathcal{K}} f \|_{\Hpi}^2 = \sum_{\gamma \in \Gamma} |\langle P_{\mathcal{K}} f, \lambda_{\Gamma}^{\sigma} (\gamma) \delta_e \rangle|^2 = \sum_{\gamma \in \Gamma} |\langle f, \lambda_{\Gamma}^{\sigma} (\gamma) P_{\mathcal{K}} \delta_e \rangle|^2, \quad f \in \mathcal{K} ,
\]
showing that $\lambda^{\sigma}_{\Gamma} (\Gamma) P \delta_{e}$ is a Parseval frame for $\mathcal{K}$.
Since $U$ is unitary, the system $ \pi (\Gamma) U^{*} P \delta_e = U^* \lambda^{\sigma}_{\Gamma} (\Gamma) P \delta_{e}$
is a Parseval frame for the span $[\pi(\Gamma) h]$.
\end{proof}

The construction of the unitary operator in Step 1 above is standard,
e.g., see \cite{rieffel1969square, rieffel2004integrable}.
It is also used, for example, in \cite{gabardo2003frame, barbieri2015riesz, fuehr2005abstract}.

\subsection{From separating vectors to orthonormal sequences}

The following result complements Proposition \ref{prop:cyclic_tight-frame}
with a similar result for separating vectors and orthonormal sequences.
In contrast to Proposition \ref{prop:cyclic_tight-frame},
the result requires the assumption that Kleppner's condition is satisfied.

\begin{proposition} \label{prop:separating_riesz}
Let $(\pi, \Hpi)$
be a square-integrable $\sigma$-representation of a countable discrete group $\Gamma$.
Suppose that $(\Gamma, \sigma)$ satisfies Kleppner's condition and
 that $\pi (\Gamma)''$ admits a separating vector.
Then there exists  $g \in \mathcal{H}_{\pi}$ such that $\CS$
forms an orthonormal sequence in $\mathcal{H}_{\pi}$.
\end{proposition}
\begin{proof}
For an arbitrary $\eta \in \Hpi$,
the map $D_{\eta, \Gamma}  : \dom(D_{\eta, \Gamma}) \to \mathcal{H}_{\pi}$
is a closed, densely defined operator by Proposition \ref{prop:denselydefined}.
The proof will be split into three steps:

\textbf{Step 1.} (Auxiliary operator $\widetilde{\pi}(c)$).
For a fixed $c \in \ell^2 (\Gamma)$, consider the auxiliary operator
$\widetilde{\pi}(c) : \dom(\widetilde{\pi}(c)) \to \mathcal{H}_{\pi}$,
with domain
\[
\dom(\widetilde{\pi} (c) ) :=
\bigg\{ g \in \mathcal{H}_{\pi} \;   | \; \exists f \in \mathcal{H}_{\pi} \; : \; \sum_{\gamma \in \Gamma} c_{\gamma} \langle \pi (\gamma) g, h \rangle \
= \langle f, h \rangle, \; \forall h \in \mathcal{B}_{\pi} \bigg\},
\]
defined by  $\widetilde{\pi} (c) g = D_{g, \Gamma} c$.
Note that $\mathcal{B}_{\pi} \subseteq \dom(\widetilde{\pi} (c))$
and hence $\widetilde{\pi} (c)$ is densely defined.
A similar argument as in part (ii) of
Proposition \ref{prop:denselydefined}
 shows that $\widetilde{\pi} (c)$ is a closed operator.

\textbf{Step 2.} ($D_{\eta, \Gamma}$ is injective for separating $\eta$).
We show that $D_{\eta, \Gamma}$ is injective if $\eta \in \Hpi$ is separating
for $\pi(\Gamma)''$.
For this,
let $c \in \dom(D_{\eta, \Gamma})$ be such that $D_{\eta, \Gamma} c = 0$.
Then, $\eta \in \dom(\widetilde{\pi}(c))$.
Let $T \in \pi(\Gamma)'$, $g \in \dom(\widetilde{\pi}(c))$ and $h \in \Bpi$ be arbitrary. Then
\[
\langle T \widetilde{\pi}(c) g, h \rangle
=
\sum_{\gamma \in \Gamma} c_{\gamma} \langle \pi(\gamma)  g, T^* h \rangle
=
\sum_{\gamma \in \Gamma} c_{\gamma} \langle \pi(\gamma) T g, h \rangle
=
\big\langle \widetilde{\pi}(c) T g, h \big\rangle,
\]
and $Tg \in \dom(\widetilde{\pi}(c))$. Hence, $\widetilde{\pi}(c) Tg = T \widetilde{\pi}(c)g$ by density of $\mathcal{B}_{\pi}$.
Let $\widetilde{\pi} (c) = J |\widetilde{\pi} (c)|  $ be the polar decomposition of $\widetilde{\pi} (c)$.
Then also $T J = J T$  by Borel functional calculus, see, e.g., \cite[Theorem 6.1.11]{kadison1986fundamentals2}, and thus $J \in \pi (\Gamma)''$. Since $0 = D_{\eta, \Gamma} c = \widetilde{\pi} (c) \eta = |(\widetilde{\pi}(c))^*|J \eta$, it follows that both $J \eta \in \overline{\mathcal{R} (\widetilde{\pi} (c))}$ and $J \eta \in
\mathcal{N}(|(\widetilde{\pi} (c))^*|) = \mathcal{R}(\widetilde{\pi}(c))^{\perp}$,
and thus $J\eta = 0$. The separateness of $\eta \in \mathcal{H}_{\pi}$ yields $J = 0$, hence $\widetilde{\pi} (c) = 0$.
In particular,
\[
\sum_{\gamma \in \Gamma} c_{\gamma} \langle \pi(\gamma) g, h \rangle = 0
\]
for all $g, h \in \Bpi$.
By Proposition \ref{prop_uniq_op}, it follows that $c = 0$.
Thus $D_{\eta, \Gamma}$ is injective.

\textbf{Step 3.} (Isometric intertwiner).
Since $D_{\eta, \Gamma} : \dom(D_{\eta, \Gamma}) \to \Hpi$
intertwines $\lambda_{\Gamma}^{\sigma}$ and $\pi$,
it follows by Proposition \ref{prop:intertwiner}
that $U : \ell^2 (\Gamma) = \mathcal{N} (D_{\eta, \Gamma})^{\perp} \to \overline{\mathcal{R} (D_{\eta, \Gamma})}$
in the polar decomposition $D_{\eta, \Gamma} = U |D_{\eta, \Gamma}|$
intertwines $\lambda^{\sigma}_{\Gamma}$ and $\pi$.
Therefore, the system
$ \pi(\Gamma) U \delta_e =  U \lambda^{\sigma}_{\Gamma} (\Gamma) \delta_e $ is an orthonormal basis for $\overline{\mathcal{R} (D_{\eta, \Gamma})}$,
thus an orthonormal sequence in $\mathcal{H}_{\pi}$.
\end{proof}

\section{Expansions in the von Neumann algebra}
\label{sec_exps}

\subsection{Expansions} \label{sec:expansions}
The following theorem provides, under Kleppner's condition, a Fourier-type series expansion
for every operator in $\pi(\Gamma)''$.

\begin{theorem}
\label{th_series}
Let $(\pi, \Hpi)$ be a square-integrable $\sigma$-representation of a countable discrete group $\Gamma$.
Suppose that $(\Gamma, \sigma)$ satisfies Kleppner's condition.
Then, for every operator $T \in \pi(\Gamma)''$, there exists a unique
$c \in \ell^2(\Gamma)$ such that $T = \pi(c)$ on $\mathcal{B}_{\pi}$, i.e.,
\begin{align}
\label{eq_Tf}
T f = \sum_{\gamma \in \Gamma} c_\gamma \pi(\gamma) f, \quad f \in \Bpi.
\end{align}
\end{theorem}
\begin{proof}
The uniqueness claim follows from Proposition \ref{prop_uniq_op}.
For the existence claim, consider the space
\begin{align}
\mathcal{A} := \left\{  \pi(c) = \sum_{\gamma \in \Gamma} c_{\gamma} \pi(\gamma) \; \bigg| \;  c \in c_{00} (\Gamma)
\right\} \subset \pi(\Gamma)'',
\end{align}
where $c_{00}(\Gamma) \subset \mathbb{C}^{\Gamma}$ are finite sequences on $\Gamma$.
The space $\mathcal{A}$ is a self-adjoint algebra containing $\pi(\Gamma) \subset \mathcal{U}(\Hpi)$.
By von Neumann's density theorem, the von Neumann algebra $\pi(\Gamma)''$
is the SOT closure of $\mathcal{A}$.
To provide \eqref{eq_Tf} for arbitrary $T \in \pi(\Gamma)''$, we first construct a vector-valued orthonormal sequence.

\textbf{Step 1.} (Existence of vector-valued orthonormal sequence)
Let $\eta = (\eta_k)_{ k \in \mathbb{N}}$
be a sequence of vectors $\eta_k \in \Hpi$
such that $\{\eta_k  :  k \in \mathbb{N} \}$ is norm dense in $\Hpi$.
Consider the direct sum $\newH := \bigoplus_{n \in \mathbb{N}} \Hpi = \ell^2(\mathbb{N}, \Hpi)$
and the associated direct sum $\sigma$-representation $(\newpi, \newH)$ of $\Gamma$, given by
\begin{align*}
\newpi(\gamma) f = (\pi(\gamma) f_k)_{k \in \mathbb{N}}, \quad f = (f_k)_{k \in \mathbb{N}} \in \newH.
\end{align*}
The associated von Neumann algebra $(\newpi(\Gamma))''$ consists of operators
$T \in \mathcal{B}(\newH)$ acting as
\begin{align}\label{eq_aaaaa}
T (f_k)_{k \in \mathbb{N}} := (A f_k)_{k \in \mathbb{N}},
\end{align}
for some $A \in \pi(\Gamma)''$.

We claim that $\eta = (\eta_k)_{k \in \mathbb{N}}$ is a separating vector for $(\newpi(\Gamma))''$. Indeed, if $T \in (\newpi(\Gamma))''$ annihilates $\eta$, then, for $A$ as in \eqref{eq_aaaaa},
$A \eta_k = 0$ for all $k \in \mathbb{N}$, and, by density, $A=0$, which implies $T=0$.

The space of Bessel vectors $\mathcal{B}_{\oplus \pi}$ of the direct sum $(\newpi, \newH)$
is norm dense in $\newH$ since it contains
$\{ (\delta_{j,k} h)_{k \in \mathbb{N}} \; : \; h \in \Bpi,\; j \in \mathbb{N} \}$.
Therefore,  Proposition \ref{prop:separating_riesz} is applicable to obtain a vector
$g = (g_k)_{k \in \mathbb{N}} \in \newH$ such that $(\newpi(\Gamma) )g$ is orthonormal in $\newH$.
Hence,
\begin{align} \label{eq:vector-valued_orthonormal}
\norm{c}_2^2 = \Bignorm{\sum_{\gamma \in \Gamma} c_\gamma \pi(\gamma) g}_{\newH}^2
= \sum_{k \in \mathbb{N}} \Bignorm{\sum_{\gamma \in \Gamma} c_\gamma \pi(\gamma) g_k}_{\Hpi}^2
\end{align}
for all $c \in \ell^2(\Gamma)$, and, in particular,
\begin{align} \label{eq:vector-valued_norm1}
\norm{g}^2_{\newH}=
\sum_{k \in \mathbb{N}} \| g_k \|_{\Hpi}^2 = 1.
\end{align}

\textbf{Step 2.} (Strong closure of $\mathcal{A}$).
Let $T \in \pi(\Gamma)''$. By von Neumann's and Kaplansky's density theorem,
there exists a bounded net $(T_{\alpha} )_{\alpha \in \Lambda}$
of operators $T_{\alpha} \in
\mathcal{A}$ such that $T_{\alpha} \sotconv T$.
Let $g = (g_k)_{  k \in \mathbb{N}} \in \newH$ be as in Step 1 satisfying \eqref{eq:vector-valued_orthonormal} and \eqref{eq:vector-valued_norm1}.
Select sequences $c^{(\alpha)} \in c_{00} (\Gamma) \subset \ell^2(\Gamma)$ such that $T_{\alpha} = \pi(c^{(\alpha)})$.
Then, for each $\alpha \in \Lambda$,
\begin{align*}
\norm{c^{(\alpha)}}_2 &=
\sum_{k \in \mathbb{N}} \Bignorm{\sum_{\gamma \in \Gamma} c^{(\alpha)}_\gamma \pi(\gamma) g_k}_{\Hpi}^2
=
\sum_{k \in \mathbb{N}} \bignorm{T_{\alpha} g_k}_{\Hpi}^2 \\
&\leq \norm{T_{\alpha}}_{op}^2 \sum_{k \in \mathbb{N}} \norm{g_k}_{\Hpi}^2 \leq \sup_{\alpha' \in \Lambda} \norm{T_{\alpha'}}_{op}^2 < \infty.
\end{align*}
By the Banach-Alaoglu theorem, we may pass to a subnet and assume that $c^{(\alpha)} \weakconv c$ for some $c \in \ell^2(\Gamma)$.

Let $f \in \mathcal{B}_\pi$ and $h \in \Hpi$ be arbitrary. Then $(\ip{\pi(\gamma) f}{h})_{\gamma \in \Gamma} \in \ell^2(\Gamma)$, and, thus,
\begin{align*}
\ip{T_{\alpha} f}{h} = \sum_{\gamma \in \Gamma} c^{(\alpha)}_\gamma \ip{\pi(\gamma) f}{h}
\longrightarrow \sum_{\gamma \in \Gamma} c_\gamma \ip{\pi(\gamma) f}{h}.
\end{align*}
On the other hand $\ip{T_{\alpha} f}{h} \longrightarrow \ip{Tf}{h}$. Hence, $\pi(c) = T$ on $\mathcal{B}_\pi$, as desired.
\end{proof}

\subsection{Coherent Riesz sequences are generated by separating vectors}
As a first application of Theorem \ref{th_series}, we show the following.
\begin{proposition}
\label{prop_riesz_sep}
Let $(\pi, \Hpi)$ be a square-integrable $\sigma$-representation of a countable discrete group $\Gamma$.
Suppose that $(\Gamma, \sigma)$ satisfies Kleppner's condition. If $\pi(\Gamma)g$ is a Riesz sequence in $\Hpi$, then $g$ is separating for $\pi(\Gamma)''$.
\end{proposition}
\begin{proof}
Suppose that $\pi(\Gamma)g$ is a Riesz sequence in $\Hpi$ and assume that $T \in \pi(\Gamma)''$ annihilates $g$. By Theorem \ref{th_series}, there exists a sequence $c \in \ell^2(\Gamma)$ such that $T=\pi(c)$. Since $\pi(\Gamma)g$ is a Riesz sequence, we have $g \in \mathcal{B}_\pi$, and, therefore,
\begin{align*}
0= \norm{T g}^2_{\Hpi} = \Bignorm{\sum_{\gamma \in \Gamma} c_\gamma \pi(\gamma) g}_{\Hpi}^2 \asymp \norm{c}_{\ell^2(\Gamma)}^2.
\end{align*}
Thus $c = 0$, and, therefore, $T = 0$, as desired.
\end{proof}

\subsection{Doubly invariant subspaces}
As a second application of Theorem \ref{th_series}, we show that $\Hpi$ does not admit so-called doubly invariant subspaces.
\begin{proposition}
\label{prop_no_double}
Let $(\pi, \Hpi)$ be a square-integrable $\sigma$-representation of a countable discrete group $\Gamma$.
Suppose that $(\Gamma, \sigma)$ satisfies Kleppner's condition.
Let $\mathcal{K} \leq \Hpi$ be a closed subspace that is invariant under $\pi(\Gamma)$ and $\pi(\Gamma)'$. Then
$\mathcal{K}=\{0\}$ or $\mathcal{K}=\Hpi$.
\end{proposition}
\begin{proof}
Consider the orthogonal projection $P_{\mathcal{K}} : \Hpi \to \Hpi$
onto $\mathcal{K}$. Since $\mathcal{K}$ is $\pi(\Gamma)'$-invariant,
it follows that
$P_{\mathcal{K}} \in \pi(\Gamma)''$ by the projection lemma.
Theorem \ref{th_series} then yields a unique
sequence $c \in \ell^2(\Gamma)$ such that
\begin{align} \label{eq:expansion_doubly}
P_{\mathcal{K}} = \pi (c) =  \sum_{\gamma' \in \Gamma} c_{\gamma'} \pi(\gamma')
\end{align}
as an operator on $\Bpi$.
Since $\mathcal{K}$ is also $\pi(\Gamma)$-invariant, it follows also that $P_{\mathcal{K}} \in \pi(\Gamma)'$.
Therefore $P_{\mathcal{K}} = \pi(\gamma) P_{\mathcal{K}} \pi(\gamma)^*$ for all $\gamma \in \Gamma$.
By Lemma \ref{lem:conjugate},
\begin{align*}
 \pi(\gamma) P_{\mathcal{K}}
\pi(\gamma)^{*}
= \sum_{\gamma' \in \Gamma} (\vartheta_{\Gamma}^{\sigma} (\gamma) c)_{\gamma'}  \pi(\gamma')
\end{align*}
The uniqueness of the expansion \eqref{eq:expansion_doubly} shows that
\begin{align} \label{eq:sequence_constantconjugacy2}
c_{\gamma'} = (\vartheta_{\Gamma}^{\sigma} (\gamma) c)_{\gamma'}
= \overline{\sigma(\gamma^{-1}, \gamma')} \sigma(\gamma^{-1} \gamma' \gamma, \gamma^{-1}) c_{\gamma^{-1} \gamma' \gamma}
\end{align}
for all $\gamma, \gamma' \in \Gamma$.
Thus $|c|$ is constant on conjugacy classes.
We now use Kleppner's condition together with the fact that $c \in \ell^2 (\Gamma)$, as in Steps 3 and 4 of the proof of Proposition \ref{prop_uniq_op},
to conclude that $c_\gamma=0$, for $\gamma \in \Gamma \setminus \{e\}$.
This shows that either $P_{\mathcal{K}} = 0$ or $P_{\mathcal{K}} = I_{\Hpi}$,
as claimed.
\end{proof}

\begin{remark}
Proposition \ref{prop_no_double} shows that, under Kleppner's condition, the center $\pi (\Gamma)'' \cap \pi(\Gamma)'$ of the algebra $\pi(\Gamma)''$ does not contain non-trivial projections, and thus equals $\mathbb{C} I_{\Hpi}$.
In technical terms: The von Neumann algebra $\pi(\Gamma)''$ is a \emph{factor}.
Kleppner's condition is also necessary for $\pi(\Gamma)''$ to be a factor.
Indeed, if $C_{\gamma_0}$ is a finite non-trivial $\sigma$-regular conjugacy class, then
the sequence $c \in \ell^2 (\Gamma)$ defined by
\[
c_{\gamma'} =
\begin{cases}
\sigma(\gamma, \gamma_0) \overline{\sigma(\gamma \gamma_0 \gamma^{-1}, \gamma)},
\quad & \text{if} \; \gamma' \in C_{\gamma_0}, \; \; \gamma' = \gamma \gamma_0 \gamma^{-1} \\
0, & \text{if} \; \gamma' \notin C_{\gamma_0}
\end{cases}
\]
is well-defined and satisfies $c_{\gamma'} = \overline{\sigma(\gamma^{-1}, \gamma')} \sigma(\gamma^{-1} \gamma' \gamma, \gamma^{-1}) c_{\gamma^{-1} \gamma' \gamma}$ for all $\gamma' \in C_{\gamma_0}$ and $\gamma \in \Gamma$,
and by Lemma
\ref{lem:conjugate}, one can see that
\begin{align*}
T := \sum_{\gamma' \in C_{\gamma_0}} c_{\gamma'} \pi(\gamma') \in \pi(\Gamma)''\cap \pi(\Gamma)'
\end{align*}
and that $T \notin \mathbb{C} I_{\Hpi}$. Therefore $\pi(\Gamma)''$ is not a factor. See \cite{kleppner1962structure,omland2014primeness}
for similar arguments.

For example, $\pi(\Gamma)''$ fails to be a factor for $\Gamma = \mathrm{SL}(2, \mathbb{Z})$ and $\pi$ being a holomorphic discrete series representation of $\mathrm{SL}(2, \mathbb{R})$, cf. Example \ref{sec:SL_ICC}.

\end{remark}

\section{Existence of cyclic or separating vectors} \label{sec:existence}
In this section we investigate how to produce large cyclic subspaces for $\pi(\Gamma)$. As a first step,
we investigate when the sum of two orthogonal cyclic subspaces, $[\pi(\Gamma)g_1]$
and $[\pi(\Gamma)g_2]$ is again cyclic. The following key lemma shows that this is the case, provided that the corresponding cyclic subspaces generated by the commutant algebra $\pi (\Gamma)'$, i.e.,
$[\pi(\Gamma)'g_1]$
and $[\pi(\Gamma)' g_2]$, are also orthogonal.

\begin{lemma}
\label{lemma_sum}
Let $(\pi, \Hpi)$ be a square-integrable $\sigma$-representation of a countable discrete group $\Gamma$.
Suppose $(g_k)_{k \in I}$ is a countable family of unit-norm vectors $g_k  \in \Hpi$  satisfying the following
\emph{simultaneous orthogonality conditions}
\begin{align}
\label{eq_d1}
\pi(\Gamma) g_k \perp \pi(\Gamma) g_j, \qquad k \not= j, \\
\label{eq_d2}
\pi(\Gamma)' g_k \perp \pi(\Gamma)' g_j, \qquad k \not= j.
\end{align}
Let $a \in \ell^1(I)$ with $a_k \not =0$ for all $k \in I$,
and set $g := \sum_{k \in I} a_k g_k$.
Then
\begin{align}
\label{eq_d1a}
[\pi(\Gamma) g] = \bigoplus_{k \in I} [\pi(\Gamma) g_k],
\\
\label{eq_d1b}
[\pi(\Gamma)' g] = \bigoplus_{k \in I} [\pi(\Gamma)' g_k].
\end{align}
\end{lemma}
\begin{proof}
Clearly,  $[\pi(\Gamma) g] \subseteq \bigoplus_{k \in I} [\pi(\Gamma) g_k]$. For the other
inclusion, let $k \in I$, and note that the projection $P_{\mathcal{K}_k}$ onto $\mathcal{K}_k = [\pi(\Gamma)' g_k]$ is in $\pi(\Gamma)''$ as $[\pi(\Gamma)' g_k]$ is $\pi(\Gamma)'$-invariant.
Therefore $g_k = {a_k}^{-1} P_{\mathcal{K}_k} g \in [\pi(\Gamma) g]$ for all $k \in I$.
This gives \eqref{eq_d1a}. The identity \eqref{eq_d1b} follows similarly, interchanging the roles of $\pi(\Gamma)'$ and
$\pi(\Gamma)''$. \
\end{proof}

\begin{proposition}
\label{prop_either}
Let $(\pi, \Hpi)$ be a square-integrable $\sigma$-representation of a countable discrete group $\Gamma$.
Suppose that $(\Gamma, \sigma)$ satisfies Kleppner's condition. Then
$\pi$ admits a cyclic vector or $\pi(\Gamma)''$ admits a separating vector (possibly both).
\end{proposition}
\begin{proof}
By Zorn's Lemma, we can select a family $(g_k)_{k \in I}$ of unit-norm vectors $g_k \in \Hpi$
satisfying the simultaneous orthogonality conditions \eqref{eq_d1} and \eqref{eq_d2},
and maximal with respect to that property.
The set $I$ is countable because $\Hpi$ is assumed to be separable.

Let $g := \sum_{k \in I} a_k g_k$ be as in Lemma \ref{lemma_sum},
so that \eqref{eq_d1a} and \eqref{eq_d1b} hold.

The maximality of $(g_k)_{ k \in I}$ implies that
\begin{align}
\label{eq_aaa}
[\pi(\Gamma) g]^\perp \cap [\pi(\Gamma)' g]^\perp = \{ 0 \};
\end{align}
otherwise,
we could choose a unit-norm vector $h \in [\pi(\Gamma) g]^\perp \cap [\pi(\Gamma)' g]^\perp$,
and extend the family $(g_k)_{k \in I}$. We claim that, in addition,
\begin{align}
\label{eq_bbb}
[\pi(\Gamma) g]^\perp \perp [\pi(\Gamma)' g]^\perp.
\end{align}
To see this, let $P_1$ and $P_2$ be the orthogonal projections onto $[\pi(\Gamma) g]$ and
$[\pi(\Gamma)' g]$ respectively. Then
$P_1 \in \pi(\Gamma)''$ and $P_2 \in \pi(\Gamma)'$, and therefore $P_1$ and $P_2$ commute. Hence,
by \eqref{eq_aaa},
\begin{align*}
\mathcal{R}( (I-P_1)(I-P_2) ) =
\mathcal{R}( (I-P_2) (I-P_1) ) \subseteq [\pi(\Gamma) g]^\perp \cap [\pi(\Gamma)' g]^\perp=\{0\}.
\end{align*}
Therefore $(I-P_1)(I-P_2)=0$, which implies \eqref{eq_bbb}.

Note that $[\pi(\Gamma) g]^\perp$ is $\pi(\Gamma)''$ invariant, while $[\pi(\Gamma)' g]^\perp$ is
$\pi(\Gamma)'$ invariant. As a consequence, the subspaces
\begin{align*}
\mathcal{K}_1 := [\pi(\Gamma)' \big([\pi(\Gamma)'' g]^\perp\big)], \qquad
\mathcal{K}_2 := [\pi(\Gamma)''\big([\pi(\Gamma)' g]^\perp\big)]
\end{align*}
are also orthogonal. Indeed, for $T' \in \pi(\Gamma)'$, $f_1 \in [\pi(\Gamma)'' g]^\perp$,
$T \in \pi(\Gamma)''$, $f_2 \in [\pi(\Gamma)' g]^\perp$,
the commutativity of $T$ and $T'$ implies that
$\ip{T' f_1}{T f_2} = \ip{T^* f_1}{(T')^* f_2} = 0$.
On the other hand, the subspaces $\mathcal{K}_1$ and $\mathcal{K}_2$ are doubly-invariant:
$\pi(\Gamma) \mathcal{K}_i = \pi(\Gamma)' \mathcal{K}_i=\mathcal{K}_i$ for $i=1,2$. Lemma \ref{prop_no_double} therefore implies that
$\mathcal{K}_i = \{0\}$ or $\mathcal{K}_i = \Hpi$ for $i=1,2$. The possibility $\mathcal{K}_1=\mathcal{K}_2=\Hpi$ is excluded (unless $\Hpi=\{0\}$)
because $\mathcal{K}_1 \perp \mathcal{K}_2$. Thus, either $\mathcal{K}_1 = \{0\}$,
or $\mathcal{K}_2 = \{0\}$.

If $\mathcal{K}_1=\{0\}$, then $[\pi(\Gamma)'' g]^\perp = \{0\}$, yielding
a cyclic vector: $\Hpi = [\pi(\Gamma)'' g]$. If $\mathcal{K}_2=\{0\}$, then
$[\pi(\Gamma)' g]=\Hpi$, which implies that $g$ is a separating vector for $\pi(\Gamma)''$ by the discussion in Section \ref{sec_comp_sep}.
\end{proof}
Lemma \ref{lemma_sum} and Proposition  \ref{prop_either} are simplifications and adaptions of standard results on central projections in reduced von Neumann algebras \cite[I.2, Proposition 3]{dixmier1981vonneumann}.

\section{Discrete series representations restricted to lattices}
\label{sec:discreteseries}
Let $G$ be a second countable unimodular group
and let $\Gamma \subset G$ be a lattice subgroup.
Let $(\pi, \Hpi)$ be a discrete series $\sigma$-representation of $G$,
i.e., irreducible and square-integrable.
This section is devoted to orbits of the restriction $\pi|_{\Gamma}$
of $(\pi, \Hpi)$ to $\Gamma$, i.e.,
\[
\pi(\Gamma)g = \big\{ \pi (\gamma) g \; : \; \gamma \in \Gamma \big\}
\]
for some $g \in \Hpi$.

In order to apply the results obtained in the previous sections,
it is essential that the restriction $\pi|_{\Gamma}$ be square-integrable
in the sense of \eqref{eq:squareintegrable_dense}. The following observation guarantees this.

\begin{lemma}
Let $\Gamma \subseteq G$ be a lattice and let
$(\pi, \Hpi)$ be a discrete series $\sigma$-representation of $G$.
 The Bessel vectors $\Bpi$
of the restriction $\pi|_{\Gamma}$ are norm dense in $\Hpi$.
\end{lemma}
\begin{proof}
Using the orthogonality relations \eqref{eq:ortho_rel},
choose $\eta \in \Hpi$ such that the map $C_{\eta} : \Hpi \to L^2 (G)$
is an isometry. Let $P_{\mathcal{K}} : L^2 (G) \to L^2 (G)$ be the orthogonal
projection onto the closed subspace $\mathcal{K} := C_{\eta} (\Hpi)$, so that $P_{\mathcal{K}} \in \lambda_{G}^{\sigma} (G)'$. It suffices to show
that the space of Bessel vectors of $\lambda_G^{\sigma} |_{\Gamma}$ is norm dense in $\mathcal{K}$, since, if $\lambda_G^{\sigma} (\Gamma) F$ is Bessel in $\mathcal{K}$, then the unitary map $C^*_{\eta} : \mathcal{K} \to \Hpi$ produces a Bessel system in $\Hpi$, namely $\pi(\Gamma) C_{\eta}^* F = C_{\eta}^* \lambda_G^{\sigma} (\Gamma) F$.

To show that the space of Bessel vectors of $\lambda_G^{\sigma} |_{\Gamma}$ is norm dense in $\mathcal{K}$, let $\Omega \subseteq G$ be a left fundamental domain for $\Gamma \subseteq G$
and consider the collection
\[\mathcal{S}_{\Omega} := \Span \left\{ F \in L^2 (G) \,:\, \supp F \subseteq \gamma \Omega \; \text{for some} \; \gamma \in \Gamma \right\}.\]
(Here $\Span$ denotes the set of finite $\mathbb{C}$-linear combinations.)
Since the sets $\{\gamma \Omega:\gamma \in \Gamma\}$ have disjoint supports, any $F' \in L^2 (G)$ can be written as $F' = \sum_{\gamma \in \Gamma} F' \cdot \chi_{\gamma \Omega}$, where $\chi_{\gamma \Omega}$ is the indicator function of $\gamma \Omega$ and the series is norm convergent by orthogonality.
Hence, $\mathcal{S}_{\Omega}$ is norm dense in $L^2 (G)$.
Therefore, the image space $P_{\mathcal{K}} \mathcal{S}_{\Omega}$ is dense in $\mathcal{K}$, and it remains to show that $P_{\mathcal{K}} \mathcal{S}_{\Omega}$ consists of Bessel vectors for $\lambda_G^{\sigma} |_{\Gamma}$. For this, note that if $F \in \mathcal{S}_{\Omega}$ is such that $\supp F \subseteq \gamma \Omega$ for some $\gamma \in \Gamma$, then the family $\lambda_{G}^{\sigma} (\Gamma) F$ is orthogonal in $L^2 (G)$, and thus $\lambda_{G}^{\sigma} (\Gamma) P_{\mathcal{K}} F = P_{\mathcal{K}} \lambda_G^{\sigma} (\Gamma) F$ is a Bessel sequence in $\mathcal{K}$.
Taking finite linear combinations, it follows that any element of $P_{\mathcal{K}} \mathcal{S}_{\Omega}$ is a Bessel vector for $\lambda_G^{\sigma} |_{\Gamma}$, with a finite Bessel constant depending on the coefficients. This completes the proof.
\end{proof}

\subsection{Frame bounds and density}
\label{sec_density_1}

The following proposition relates frame bounds \eqref{eq_frame_bounds}, formal dimension and co-volume.

\begin{proposition} \label{prop:density_quasi}
Let $\Gamma \subseteq G$ be a lattice and let
$(\pi, \Hpi)$ be a discrete series $\sigma$-representation of $G$ of formal dimension $d_{\pi} > 0$.
If $\CS$ admits a Bessel bound $B > 0$,
then $ d_{\pi}^{-1} \| g \|_{\mathcal{H}_{\pi}}^2 \leq B \vol (G / \Gamma)$.
Moreover, if $\CS$ also admits a lower frame bound $A > 0$,
then
\begin{align} \label{eq:framebounds_formal}
A \vol(G /\Gamma) \leq  d_{\pi}^{-1} \| g \|_{\mathcal{H}_{\pi}}^2 \leq B \vol (G / \Gamma).
\end{align}
\end{proposition}
\begin{proof}
Let $\Omega \subseteq G$ be a right fundamental domain of $\Gamma \subseteq G$.
Then
\begin{align*}
\int_G |\langle f, \pi (x) g \rangle |^2 \; d\mu_G (x)
= \sum_{\gamma \in \Gamma} \int_{\Omega} |\langle f, \pi(x \gamma ) g \rangle|^2 \; d\mu_{G} (x)
= \int_{\Omega} \sum_{\gamma \in \Gamma} |\langle \pi(x)^* f, \pi(\gamma) g \rangle|^2 \; d\mu_G (x)
\end{align*}
for $f \in \mathcal{H}_{\pi}$.
This, together with the orthogonality relations \eqref{eq:ortho_rel}, yields
\begin{align*}
d_{\pi}^{-1} \|f \|_{\mathcal{H}_{\pi}}^2 \|g\|_{\mathcal{H}_{\pi}}^2
=  \int_{\Omega} \sum_{\gamma \in \Gamma} |\langle \pi(x)^* f, \pi(\gamma) g \rangle|^2 \; d\mu_G (x).
\end{align*}
Thus, if $\CS$ is Bessel with bound $B$,
then
$ d_{\pi}^{-1} \|f \|_{\mathcal{H}_{\pi}}^2 \|g\|_{\mathcal{H}_{\pi}}^2
\leq B \int_{\Omega} \| f \|_{\mathcal{H}_{\pi}}^2 d\mu_G (x),
$
which shows the upper bound in \eqref{eq:framebounds_formal}. The desired lower bound is proven similarly.
\end{proof}

\begin{remark}
The proof of Proposition \ref{prop:density_quasi} also works for discrete subgroups
 $\Gamma \subset G$ having possibly infinite co-volume.
 However, the lower bound in \eqref{eq:framebounds_formal} shows that
 the restriction $\pi|_{\Gamma}$ admits a frame only if $\Gamma \subset G$ has finite co-volume.
 The lattice assumption is in fact even necessary for $\pi|_{\Gamma}$
 to admit a cyclic vector \cite[Corollary 2]{bekka2004square}.
\end{remark}

The idea of periodizing the orthogonality relations by means of Weil's integral formula can also be found in
\cite{cowling1991irreducibility, kuhn1992restrictions}.
Proposition \ref{prop:density_quasi} will be subsequently substantially sharpened by eliminating the frame bounds in the conclusion.

\subsection{Necessary density conditions}
\label{sec_density_2}

The following result provides necessary density conditions
for several spanning properties.
Note that Kleppner's condition is not assumed in parts (i) and (ii).

\begin{theorem} \label{thm:necessary_density}
Let $\Gamma \subseteq G$ be a lattice and let
$(\pi, \Hpi)$ be a discrete series $\sigma$-representation of $G$ of formal dimension $d_{\pi} > 0$.
\begin{enumerate}[(i)]
\item[(i)] If $\pi|_{\Gamma}$ admits a cyclic vector, then $\vol (G / \Gamma) d_{\pi} \leq 1$.

\noindent In particular, if $\pi|_{\Gamma}$ admits a frame vector, then  $\vol (G / \Gamma) d_{\pi} \leq 1$.

\item[(ii)] If $\pi|_{\Gamma}$ admits a Riesz vector, then  $\vol (G / \Gamma) d_{\pi} \geq 1$.

\item[(iii)] Suppose $(\Gamma, \sigma)$ satisfies Kleppner's condition.
If $\pi(\Gamma)''$ admits a separating vector, then $\vol (G / \Gamma)d_{\pi} \geq 1$.
\end{enumerate}
\end{theorem}
\begin{proof}
(i) Suppose first that there exists a vector $g \in \mathcal{H}_{\pi}$
such that $\CS$ is a Parseval frame for $\mathcal{H}_{\pi}$.
Then, $\vol(G/\Gamma) d_{\pi} = \|g \|_{\mathcal{H}_{\pi}}^2$, by Proposition \ref{prop:density_quasi}. Since $\CS$ is a Bessel sequence with bound $1$, necessarily $\| g \|_{\mathcal{H}_{\pi}}^2 \leq 1$.
Hence $\vol(G/\Gamma) d_{\pi} \leq 1$, as claimed. Second, if $\pi|_{\Gamma}$ admits a cyclic vector, then it also admits a Parseval frame by Proposition \ref{prop:cyclic_tight-frame}.

\noindent (ii) Suppose that $\pi|_{\Gamma}$ admits a Riesz vector. Then, by Remark \ref{rem_Riesz_on}, there also exists $g \in \mathcal{H}_{\pi}$ such that
$\CS$ is orthonormal. Hence, $\CS$ has Bessel bound $1$, and, therefore, by
Proposition \ref{prop:density_quasi}, $d_{\pi}^{-1}  \leq \vol(G/ \Gamma)$.

\noindent (iii) Finally, under Kleppner's condition, if $\pi(\Gamma)''$ admits a separating vector, then it also admits an orthonormal sequence by Proposition \ref{prop:separating_riesz}, and we can apply part (ii).
\end{proof}

The idea of relating the orthogonality relations and
the frame inequalities for proving a density theorem as Theorem \ref{thm:necessary_density}
was used in Janssen's ``classroom proof'' of the density theorem for Gabor frames \cite{janssen_density}.
The use of an auxiliary tight frame to deduce
the density condition can be found in \cite[Theorem 11.3.1]{christensen2016introduction}.
A similar combination of these ideas have been used in \cite{jakobsen2016density}.
In this article, these ideas are further refined, implying necessary conditions for completeness.
The arguments for Riesz sequences seem to be new.

\subsection{Critical density}
This section is devoted to the spanning properties of $\pi|_{\Gamma}$
for lattices possessing the \emph{critical density} $\vol(G/\Gamma) d_{\pi} = 1$.

\begin{lemma} \label{lem:critical_complete}
Let $\Gamma \subseteq G$ be a lattice and let
$(\pi, \Hpi)$ be a discrete series $\sigma$-representation of $G$ of formal dimension $d_{\pi} > 0$.
Suppose $g \in \mathcal{H}_{\pi}$ is a unit vector such that $\CS$
is an orthonormal system in $\Hpi$. Then the following are equivalent:
\begin{enumerate}[(i)]
\item The system $\CS$ is complete in $\mathcal{H}_{\pi}$.
\item $\vol(G / \Gamma)  d_{\pi} = 1$.
\end{enumerate}
\end{lemma}
\begin{proof}
That (i) implies (ii) follows from Proposition \ref{prop:density_quasi}.

Conversely, suppose that $\vol(G/ \Gamma) d_{\pi} = 1$.
Let $\Omega \subseteq G$ be a right fundamental domain of $\Gamma \subseteq G$,
and $\{f_n: n \in \mathbb{N}\}$ a norm dense subset of $\Hpi$.

Fix $n \in \mathbb{N}$.
By the orthogonality relations \eqref{eq:ortho_rel}
and the assumption $\| g \|_{\mathcal{H}_{\pi}} = 1$,
\begin{align*}
 d_{\pi}^{-1} \| f_n \|_{\mathcal{H}_{\pi}}^2 = \int_G | \langle f_n, \pi (x) g \rangle |^2 \; d\mu_G (x)
 = \sum_{\gamma \in \Gamma} \int_{\Omega} | \langle f_n, \pi (x  \gamma) g \rangle |^2 \; d\mu_G (x) .
 \end{align*}
Since $\vol(G/ \Gamma)d_{\pi} = 1$,
 \begin{align*}
0 &= d_{\pi} \int_{\Omega} \| f_n \|_{\mathcal{H}_{\pi}}^2 \; d\mu_G(x) - d_{\pi} \sum_{\gamma \in \Gamma} \int_{\Omega} | \langle f_n, \pi (x \gamma) g \rangle |^2 \; d\mu_G (x)  \\
&= d_{\pi} \int_{\Omega} \bigg( \|f_n\|_{\mathcal{H}_{\pi}}^2 - \sum_{\gamma \in \Gamma} | \langle f_n, \pi (x \gamma) g \rangle |^2 \bigg)\; d\mu_G (x). \numberthis \label{eq:nonzero_integrand}
\end{align*}
But $\sum_{\gamma \in \Gamma} | \langle f_n, \pi (x \gamma) g \rangle |^2 \leq \| f \|^2_{\mathcal{H}_{\pi}}$
for any $x \in \Omega$ by Bessel's inequality.
Thus the integrand in \eqref{eq:nonzero_integrand} is $0$ for $x \in \Omega \setminus E_n$, where $E_n \subseteq \Omega$ is a null measure set.

Since $\bigcup_{n \in \mathbb{N}} E_n$ has null measure, we can choose $x_0 \in \Omega \setminus \bigcup_{n \in \mathbb{N}} E_n$.
Therefore,
\begin{align*}
\sum_{\gamma \in \Gamma} | \langle f, \pi (x_0 \gamma) g \rangle |^2 = \| f \|^2_{\mathcal{H}_{\pi}},
\end{align*}
holds for all $f \in \{f_n : n \in \mathbb{N}\}$, and extends by density to all $f \in \Hpi$.
Replacing $f$ by $\pi(x_0)f$ gives
$
\sum_{\gamma \in \Gamma} | \langle f, \pi (\gamma) g \rangle |^2 = \| f \|^2_{\mathcal{H}_{\pi}}$,
for all $f \in \Hpi$. This shows that $\CS$ is complete.
\end{proof}

\begin{proposition} \label{prop:critical_density}
Let $\Gamma \subseteq G$ be a lattice and let
$(\pi, \Hpi)$ be a discrete series $\sigma$-representation of $G$ of formal dimension $d_{\pi} > 0$.
The following assertions are equivalent:
\begin{enumerate}[(i)]
\item The system $\CS$ is a Riesz (resp. orthonormal) basis for $\mathcal{H}_{\pi}$.
\item The system $\CS$ is a frame (resp. Parseval frame) for $\mathcal{H}_{\pi}$ with
$\vol (G / \Gamma)  d_{\pi} = 1$.
\item The system $\CS$ is a Riesz (resp. orthonormal) sequence  in $\mathcal{H}_{\pi}$ with
$\vol (G / \Gamma)  d_{\pi} = 1$.
\end{enumerate}
\end{proposition}
\begin{proof}
The implications (i) $\Rightarrow$ (ii) and (i) $\Rightarrow$ (iii) follow directly from Theorem \ref{thm:necessary_density}.

(ii) $\Rightarrow$ (i) First, we show that a Parseval
frame $\pi(\Gamma) g$ with $\vol(G/\Gamma) d_{\pi} = 1$
is an orthonormal basis for $\Hpi$. Indeed, we have $\| g \|_{\Hpi}^2 = 1$ by Proposition \ref{prop:density_quasi}, and hence
\begin{align} \label{eq:parseval_unit}
1 = \| \pi(\gamma') g \|_{\Hpi}^2 = \sum_{\gamma \in \Gamma} | \langle \pi(\gamma') g, \pi(\gamma) g \rangle|^2 = 1 + \sum_{\gamma \in \Gamma \setminus \{\gamma'\}} |\langle \pi(\gamma') g, \pi(\gamma) g \rangle|^2,
\end{align}
which shows that $\langle \pi(\gamma') g, \pi(\gamma) g \rangle = \delta_{\gamma, \gamma'}$ for all $\gamma, \gamma' \in \Gamma$.
Thus $\pi(\Gamma) g $ is an orthonormal basis for $\mathcal{H}_{\pi}$.

Second, if $\pi(\Gamma) g$ is an arbitrary frame with $\vol(G/\Gamma) d_{\pi} = 1$, then
$\pi (\Gamma) S_{g, \Gamma}^{-1/2} g$ is a Parseval frame for $\Hpi$,
and hence an orthonormal basis for $\Hpi$ by the above.
But $\CS  = S_{g,\Gamma}^{1/2} \pi(\Gamma) S_{g,\Gamma}^{-1/2} g$,
and thus $\CS$ is a Riesz basis for $\Hpi$.

(iii) $\Rightarrow$ (i) Suppose $\CS$ is a Riesz sequence in $\Hpi$. Then $S^{-1/2}_{g,\Gamma} : [\pi(\Gamma) g] \to [\pi(\Gamma)g]$ is well-defined and bounded. Hence, the system $\pi(\Gamma) S_{g,\Gamma}^{-1/2} g $ is orthonormal in $\mathcal{H}_{\pi}$
by Remark \ref{rem_Riesz_on}, thus complete by Lemma \ref{lem:critical_complete}.
As above, $\pi (\Gamma) g = S_{g,\Gamma}^{1/2} \pi(\Gamma) S_{g,\Gamma}^{-1/2} g$,
showing that $\CS$ is a Riesz basis.
Moreover, if $\CS$ itself is orthonormal, then its completeness follows directly by Lemma \ref{lem:critical_complete}.
\end{proof}

\section{Proof of the density theorem} \label{sec:density_thm}
We finally can prove the main result of the article.
\begin{theorem}
\label{th_main}
Let $\Gamma \subseteq G$ be a lattice in a second countable unimodular group $G$. Let
$(\pi, \Hpi)$ be a discrete series $\sigma$-representation of $G$
of formal dimension $d_{\pi} > 0$.
Suppose that $(\Gamma, \sigma)$ satisfies Kleppner's condition.
Then the following assertions hold:
\begin{itemize}
\item[(i)] If $\vol (G / \Gamma) d_{\pi} < 1$, then $\pi|_{\Gamma}$ admits a Parseval frame,
but $\pi(\Gamma)''$ does not admit a separating vector. (In particular, $\pi|_{\Gamma}$ does not admit a Riesz vector.)

\item[(ii)] If $\vol (G / \Gamma) d_{\pi} = 1$, then $\pi|_{\Gamma}$ admits an orthonormal basis.

\item[(iii)] If $\vol (G / \Gamma) d_{\pi} > 1$, then $\pi|_{\Gamma}$ admits an orthonormal sequence, but not a cyclic vector. (In particular, $\pi|_{\Gamma}$ does not admit a frame vector.)
\end{itemize}
\end{theorem}
\begin{proof}
\noindent (i) Assume that $\vol (G / \Gamma) d_{\pi} < 1$. Then, by Theorem \ref{thm:necessary_density},
$\pi(\Gamma)''$ does not admit a separating vector. Combining this information with Proposition \ref{prop_either}, it follows that $\pi|_{\Gamma}$ admits a cyclic vector, and by Proposition \ref{prop:cyclic_tight-frame} also a Parseval frame. The ``in particular" part also follows from Theorem \ref{thm:necessary_density}.

\noindent (ii) Assume that $\vol (G / \Gamma) d_{\pi} = 1$. By Proposition \ref{prop_either}, $\pi|_{\Gamma}$ admits either a cyclic  or separating vector. In the first case, by Proposition \ref{prop:cyclic_tight-frame}, $\pi|_{\Gamma}$ also admits
a Parseval frame $\CS$, and hence an orthonormal basis by Proposition \ref{prop:critical_density}.
In the second case, by Proposition \ref{prop:separating_riesz}, $\pi|_{\Gamma}$  admits an orthonormal sequence $\CS$, which forms an orthonormal basis by
Proposition \ref{prop:critical_density}.

\noindent (iii) Assume that $\vol (G / \Gamma) d_{\pi} > 1$. Then, by Theorem \ref{thm:necessary_density},
$\pi|_{\Gamma}$ does not admit a cyclic vector. Combining this information with Proposition \ref{prop_either}, it follows that $\pi(\Gamma)''$ admits a separating vector, and by Proposition \ref{prop:separating_riesz}, also an orthonormal sequence.
\end{proof}

\subsection{Proof of Theorem \ref{thm:intro_ICC}}
The hypotheses of Theorem \ref{thm:intro_ICC} are a particular case of Theorem \ref{th_main}. Indeed, an ICC lattice $\Gamma$ satisfies Kleppner's condition
for any cocycle $\sigma$, in particular, for $\sigma \equiv 1$.  \footnote{In part (i) of Theorem \ref{thm:intro_ICC}, the assertion that $\pi|_{\Gamma}$ does not admit a separating vector means that $\pi(\Gamma)''$ does not admit such a vector. } \QED

A far reaching generalization of Theorem \ref{thm:intro_ICC} without the ICC condition is due to Bekka
\cite{bekka2004square}; see Section \ref{sec:examples}.

\section{Examples and applications} \label{sec:examples}

\subsection{The density theorem for semisimple Lie groups}
For certain center-free semisimple Lie groups,
a lattice is automatically ICC,
and hence Kleppner's condition is satisfied.
For reference purposes, we state Theorem \ref{th_main} in this setting.

\begin{theorem}
\label{th_semi}
Let $G$ be a center-free connected semisimple real Lie group all of whose connected, normal, compact subgroups are trivial.\footnote{In the jargon of semisimple Lie groups, a group all of whose connected, normal, compact subgroups are trivial is sometimes referred to as \emph{a group without compact factors}.}
Let $(\pi, \Hpi)$ be a discrete series $\sigma$-representation of $G$ of formal dimension $d_{\pi} > 0$.
Let $\Gamma \subseteq G$ be a lattice. Then
\begin{itemize}
\item[(i)] If $\vol (G / \Gamma) d_{\pi} < 1$, then $\pi|_{\Gamma}$ admits a Parseval frame,
but $\pi(\Gamma)''$ does not admit a separating vector. (In particular, $\pi|_{\Gamma}$ does not admit a Riesz vector.)

\item[(ii)] If $\vol (G / \Gamma) d_{\pi} = 1$, then $\pi|_{\Gamma}$ admits an orthonormal basis.

\item[(iii)] If $\vol (G / \Gamma) d_{\pi} > 1$, then $\pi|_{\Gamma}$ admits an orthonormal sequence, but not a cyclic vector. (In particular, $\pi|_{\Gamma}$ does not admit a frame vector.)
\end{itemize}

\end{theorem}
\begin{proof}
A connected semisimple Lie group is unimodular, see, e.g., \cite[Corollary 8.31]{knapp2002lie}.
Under the additional hypotheses, a lattice $\Gamma$ is an ICC group by
\cite[Lemma 3.3.1]{goodman1989coxeter} or \cite[Theorem 2]{bekka2004square}. Therefore,
$(\Gamma, \sigma)$ satisfies Kleppner's condition, and the conclusion follows from
Theorem \ref{th_main}.
\end{proof}

As we show in Example \ref{sec:SL_ICC}, the conclusion of Theorem \ref{th_semi} may fail when the center of the group is non-trivial.
A more general version of Theorem \ref{th_semi}, that does not require the ICC condition, was derived by Bekka \cite{bekka2004square},
and applies to semisimple Lie groups with a possibly non-trivial center \cite[Theorem 2]{bekka2004square},
and to a class of algebraic groups over more general fields.
Theorem \ref{th_semi} can also be phrased more generally for such algebraic groups, provided they have a trivial center.

We now illustrate an important instance of Theorem \ref{th_semi}.

\begin{example}
\label{sec_bergman}
The group $G = \mathrm{PSL}(2, \mathbb{R}) = \mathrm{SL}(2, \mathbb{R}) / \{-I, I\}$
is a connected simple Lie group with trivial center \cite{fulton1991representation, varadarajan1984lie},
and acts on the upper half plane \[\mathbb{C}^+ = \{z \in \mathbb{C} : \Im (z) > 0\}\]
through Moebius transforms as
\[
G \times \mathbb{C}^+ \ni \big(
\begin{pmatrix}
a & b \\
c & d
\end{pmatrix},
z
\big)
\mapsto \frac{a z + b}{cz + d} \in \mathbb{C}^+.
\]
The measure $d\mu (z) = (\Im (z))^{-2} dx dy$, where
$z = x+iy$ and
$dxdy$ is the Lebesgue measure on $\mathbb{C}^+$,
is $G$-invariant.
Let $\mathrm{PSO}(2,\mathbb{R}) := \mathrm{SO}(2,\mathbb{R}) / \{-I, I\}$
be the compact subgroup of rotations. We use the diffeomorphism,
\begin{align}\label{eq_diff_i}
G / \mathrm{PSO}(2, \mathbb{R}) &\to \mathbb{C}^+,
\\
[ \newg ] &\mapsto m \cdot i,
\end{align}
to fix a Haar measure on $G / \mathrm{PSO}(2, \mathbb{R})$,
and equip $\mathrm{PSO}(2, \mathbb{R})$ with a normalized Haar measure $\mu_{T}$
of total measure $1$. This fixes the Haar measure $\mu_G$ on $G$
as $d\mu_G \simeq d\mu d\mu_T$. With this normalization,
for measurable $E \subseteq \mathbb{C}^+$,
\begin{align}\label{eq_norm_haar}
\mu_G\left( \{m \in G: m \cdot i \in E \}\right) = \mu (E).
\end{align}
In the remainder of this article, the Haar measure on $G = \mathrm{PSL}(2,\mathbb{R})$
is always assumed to have this normalization.

For $\alpha > 1$,
define the measure
$d\mu_{\alpha} (z) = (\Im (z))^{\alpha -2} dxdy$
and the weighted Bergman space of holomorphic functions
$
A^2_{\alpha} (\mathbb{C}^+) := \mathcal{O} (\mathbb{C}^+) \cap L^2 (\mathbb{C}^+, d\mu_{\alpha})$, equipped
with norm
\begin{align} \label{eq:bergman_norm}
\| f \|^2_{A^2_{\alpha}} =
\int_{\mathbb{C}^+} |f(z)|^2 \; d\mu_{\alpha} (z).
\end{align}

Define $j : \mathrm{SL}(2, \mathbb{R}) \times \mathbb{C}^+ \to \mathbb{C} \setminus \{0\}$ by
\begin{align}
\label{eq_pialpha}
j(\newg, z) = (cz + d)^{-1}, \quad \newg = \begin{pmatrix}
a & b \\
c & d
\end{pmatrix}, z \in \mathbb{C}^+.
\end{align}
Then $j$ satisfies
$
j(\newg_1 \newg_2, z) = j(\newg_1, \newg_2 z) j (\newg_2, z)
$
for all $\newg_1, \newg_2 \in \mathrm{SL}(2, \mathbb{R})$ and $z \in \mathbb{C}^+$.

Let $z^\alpha$ be defined with respect to the principal branch of the argument:
$\arg(z) \in (-\pi, \pi]$. Since $j(\newg, z) \in \mathbb{C}\setminus\mathbb{R}$,
we can form $j(\newg, z)^\alpha$, and
\[j(\newg_1 \newg_2, z)^\alpha = \lambda(\newg_1, \newg_2, \alpha, z)
j(\newg_1, \newg_2 z)^\alpha j (\newg_2, z)^\alpha\] for a unimodular function
$\lambda(\newg_1, \newg_2, \alpha, z)$. The analyticity of $j(\newg, z)$ on $z$, implies that $\lambda(\newg_1, \newg_2, \alpha, z)=\lambda(\newg_1, \newg_2,\alpha)$ is independent of $z$.
A projective unitary representation  $(\pi'_{\alpha}, A^2_{\alpha} (\mathbb{C}^+))$ of $\mathrm{SL}(2, \mathbb{R})$ is
therefore given by
\begin{align} \label{eq:holomorphicdiscreteseries}
(\pi'_{\alpha}(\newg) f )(z) = j(\newg^{-1}, z)^{\alpha} f (\newg^{-1} \cdot z), \quad \newg \in \mathrm{SL}(2, \mathbb{R}), \; z \in \mathbb{C}^+.
\end{align}
Let $\tau: G \to \mathrm{SL}(2, \mathbb{R})$ be a Borel cross-section of the quotient map, i.e., a Borel measurable function that chooses a representative; see \cite[Lemma 1.1.]{mackey1952induced}
or \cite[Chapter 5]{varadarajan1985geometry}. Since $j(-\newg, z)=-j(\newg, z)$,
$\pi_\alpha := \pi'_\alpha \circ \tau$ defines a projective unitary representation of $G$ on
$A^2_{\alpha} (\mathbb{C}^+)$, the so-called \emph{holomorphic discrete series $\sigma$-representation}.
For any $\alpha > 1$, $(\pi_{\alpha}, A^2_{\alpha} (\mathbb{C}^+))$ is irreducible and
square-integrable of formal dimension
\[
d_{\pi_{\alpha}} = \frac{\alpha - 1}{4\pi}.
\]
See \cite{radulescu1998berezin, robert1983introduction} for the details.

Lattices $\Gamma \subseteq G$ are known as \emph{Fuchsian groups}.
By the normalization \eqref{eq_norm_haar}, we have $\vol(G/\Gamma)=\mu(D)$, where $D \subseteq \mathbb{C}^+$ is a so-called \emph{Dirichlet fundamental domain} for $\Gamma$, that provides the tessellation $\mathbb{C}^+ = \bigcup_{\gamma \in \Gamma} \gamma D$, up to sets of null measure.

According to Theorem \ref{th_main},
the existence of a function $g \in A^2_{\alpha} (\mathbb{C}^+)$ such that
$\pi_{\alpha}(\Gamma) g$ is complete in (resp. frame for, resp. Parseval frame for) $A^2_{\alpha} (\mathbb{C}^+)$
is equivalent to the condition
\begin{equation}
\label{eq_compl_sl2r}
\vol(G/\Gamma) \leq \frac{4\pi}{\alpha - 1},
\end{equation}
while the existence of a Riesz sequence $\pi_{\alpha}(\Gamma) g$ (resp. orthonormal sequence, resp. $g$ separating vector) in $A^2_{\alpha} (\mathbb{C}^+)$
is equivalent to the condition
\begin{equation}
\label{eq_sep_sl2r}
\vol(G/\Gamma) \geq \frac{4\pi}{\alpha - 1}.
\end{equation}
For examples of Fuchsian groups, and formulae for their co-volume, see \cite{beardon1983geometry}.
\end{example}

The following example demonstrates that Kleppner's condition (or the ICC condition) cannot be removed as an assumption in Theorem \ref{th_main}.
\begin{example} \label{sec:SL_ICC}

Let $G = \mathrm{SL}(2, \mathbb{R})$, with center $Z(G) = \{-I, I\}$.
For $\alpha > 1$,
the group $G$ acts on the Bergman space $A^2_{\alpha} (\mathbb{C}^+)$
by the representation $\pi'_{\alpha}$ whose action is given by
\eqref{eq:holomorphicdiscreteseries}.
Equip $Z(G)$ with the counting measure and $G / Z(G) = \mathrm{PSL}(2, \mathbb{R})
$ with the Haar measure $\mu_{G/Z}$ normalized as in Example \ref{sec_bergman}.
The Haar measure on $G$ is then fixed by Weil's formula: $d\mu_G \simeq d\mu_{G/Z} d\mu_Z$.
By the orthogonality relations of the holomorphic discrete series of $\mathrm{PSL}(2, \mathbb{R})$, it follows then that, for $f \in A^2_{\alpha} (\mathbb{C}^+)$,
\[
\int_G | \langle f, \pi'_{\alpha}(x) f \rangle|^2 \; d\mu_G (x) = \int_{G/Z(G)} \sum_{ Z(G)}
|\langle f, \pi'_{\alpha} (\dot{x}) f \rangle |^2 \; d\mu_{G/Z} (xZ) = (2 \cdot d^{-1}_{\pi_{\alpha}} )\| f \|_{\Hpi}^4,
\]
where $d_{\pi_{\alpha}} := (\alpha - 1)/(4\pi)$ as in Example \ref{sec_bergman}.
Thus $\pi'_{\alpha}$ is a discrete series representation of $G$ of
formal dimension $d'_{\pi_{\alpha}} = (\alpha - 1)/(8\pi)$.

Let $\Gamma \subseteq G$ be a lattice such that $Z(G) \subset \Gamma$, and
$\Omega \subset G$ a right fundamental domain. Denote by $p : G \to G/Z(G)$
the canonical projection, and $\widetilde{\Gamma} = p(\Gamma)$. As $Z(G) \subset \Gamma$,
$\chi_{\Omega}(x)+\chi_{\Omega}(-x)=\chi_{p(\Omega)}(xZ)$, and, therefore,
Weil's formula gives
\begin{align*}
\vol(G/\Gamma) &= \int_G \chi_{\Omega} (x) \; d\mu_G (x)
= \int_{G / Z(G)} \chi_{p(\Omega)}(xZ) \; d\mu_{G/Z} (xZ)\\
&= \mu_{G/Z} (p(\Omega)) = \vol(\mathrm{PSL}(2, \mathbb{R}) / \widetilde{\Gamma}),
\end{align*}
as $p(\Omega)$ is a fundamental domain for $\widetilde{\Gamma}$ in $\mathrm{PSL}(2, \mathbb{R})$.

Consider the representation $\pi_\alpha$ from Example \ref{sec_bergman}.
Since $\pi'_\alpha(-I)=\pm I$, for any $g \in A^2_{\alpha} (\mathbb{C}^+)$,
\begin{align*}
[\pi'_{\alpha} (\Gamma) g] = [\pi_{\alpha} (\widetilde{\Gamma}) g].
\end{align*}
We conclude that
there exists $g \in A^2_{\alpha} (\mathbb{C}^+)$ such that $\pi'_{\alpha} (\Gamma) g$
is complete  if and only if
\begin{align} \label{completeness_SL}
\vol(G/\Gamma) = \vol(\mathrm{PSL}(2,\mathbb{R}) / \widetilde{\Gamma}) \leq \frac{4\pi}{\alpha - 1}
=  \frac{1}{2} (d'_{\pi_{\alpha}})^{-1},
\end{align}
or, equivalently, $\vol(G/\Gamma) d'_{\pi_{\alpha}} \leq 1/2$.
(This conclusion follows also from Bekka's result \cite[Example 1]{bekka2004square},
where a different normalizations of the Haar measure is used.)

Therefore, the completeness part of Theorem \ref{th_main} fails for $G$ and $\Gamma$. Of course,
$(\Gamma,\sigma)$ does not satisfy Kleppner's condition, as the central element $-I \in G$ has a finite conjugacy class.

Second, note that there does not exist a Riesz sequence in $A^2_{\alpha} (\mathbb{C}^+)$
of the form $\pi'_{\alpha} (\Gamma) g$, regardless of the value of $\vol(G/\Gamma)$, as the (indexed) system
$\pi'_{\alpha} (\Gamma) g$ is always linearly dependent: $\pi'_{\alpha} (I) g = g = \pm \pi'_{\alpha} (-I)g$.
Hence, also in that respect, the conclusion of Theorem \ref{th_main} fails for $G$ and $\Gamma$.
\end{example}

\subsubsection{Perelomov's uniqueness problem} \label{sec:perelomov_uniqueness}
A set of points $\Lambda \subseteq \mathbb{C}^+$ is called a \emph{set of uniqueness} for
the Bergman space $A^2_{\alpha} (\mathbb{C}^+)$ if the only function $f \in A^2_{\alpha} (\mathbb{C}^+)$
that vanishes identically on $\Lambda$ is the zero function.
Perelomov \cite{perelomov1973coherent} studied this question when $\Lambda$ is the orbit of a point $w \in \mathbb{C}^+$ through a Fuchsian group $\Gamma$
in $G = \mathrm{PSL}(2,\mathbb{R})$.\footnote{Perelomov formulates his results on the unit disk.}
 The link with lattice orbits of $\pi_{\alpha}$ is provided by the special choice of vector
$\repkw (z)=2^{\alpha-2}\pi^{-1}(\alpha-1)i^\alpha (z-\overline{w})^{-\alpha}$,
which has the \emph{reproducing property}:
\begin{align}\label{eq_rep_for}
f(\newg \cdot w) = c_\alpha (cw+d)^{-\alpha} \langle f, \pi_{\alpha}(\newg) \repkw\rangle_{A^2_\alpha}, \qquad f \in A^2_{\alpha} (\mathbb{C}^+), \newg \in \mathrm{PSL}(2, \mathbb{R}),
\end{align}
where $c_\alpha \in \mathbb{T}$ is a unimodular constant and the notation of \eqref{eq_pialpha} is used.
Hence, $\Lambda=\Gamma w$ is a set of uniqueness for $A^2_{\alpha} (\mathbb{C}^+)$ if and only if
$\pi_{\alpha} (\Gamma) \repkw$ is complete in $A^2_{\alpha} (\mathbb{C}^+)$. Perelomov \cite{perelomov1973coherent} showed that this is the case if
\begin{align}
\label{eq_perelomov}
\#F_w \vol(G/\Gamma) < \frac{4\pi}{\alpha - 1},
\end{align}
where
$
F_w = \left\{\gamma \in \Gamma : \gamma \cdot w = w \right\}
$
is the stabilizer subgroup of $w$.\footnote{In \cite[Theorems 3 and 4]{perelomov1973coherent} Perelomov implicitly assumes that $\#F_w=1$, the general case follows after some minor adaptations, as explained in \cite[Theorem 5.1]{kelly}. The case $\Gamma=\mathrm{PSL}(2,\mathbb{Z})$ is proved independently in \cite{klauder1994wavelets}, after observing that the physically-motivated restrictions the authors impose on $\alpha$ play no role in the argument.}

When $\#F_w=1$, the sufficient condition for the completeness of $\pi_{\alpha}(\Gamma) \repkw$ in $A^2_{\alpha} (\mathbb{C}^+)$ \eqref{eq_perelomov} almost matches \eqref{eq_compl_sl2r}, which is necessary for the completeness of \emph{any} orbit $\pi_{\alpha}(\Gamma) g$. For $\#F_w>1$, a necessary condition for the completeness of $\pi_{\alpha}(\Gamma) \repkw$ in $A^2_{\alpha} (\mathbb{C}^+)$ almost matching \eqref{eq_perelomov} was proved by Kelly-Lyth \cite[Theorem 5.4]{kelly}: if $\Lambda$ is a uniqueness set for $A^2_{\alpha} (\mathbb{C}^+)$, then
\begin{align}
\label{eq_kelly_lyth}
\#F_w \vol(G/\Gamma) \leq \frac{4\pi}{\alpha - 1}.
\end{align}
Thus, while $(\pi_{\alpha}, A^2_{\alpha} (\mathbb{C}^+))$ admits a cyclic vector $g$ if and only if
$\vol(G/\Gamma) \leq \frac{4\pi}{\alpha - 1}$, in the smaller range
$\vol(G/\Gamma) < \frac{4\pi}{\#F_w(\alpha - 1)}$ the specific choice $g=\repkw$ is possible,
and in the range $\frac{4\pi}{\#F_w(\alpha - 1)} < \vol(G/\Gamma) \leq \frac{4\pi}{\alpha - 1}$
it is not. The completeness of $\pi_{\alpha} (\Gamma)\repkw$ when
$\#F_w \vol(G/\Gamma) = \frac{4\pi}{\alpha - 1}$ has recently been shown by Jones \cite{jones2020bergman}.

Perelomov's original work also contains a necessary condition for the completeness of $\pi_{\alpha}(\Gamma) \repkw$ in $A^2_{\alpha} (\mathbb{C}^+)$, formulated in terms of the smallest weight $m_0^+$
for which the space of parabolic $\Gamma$-modular forms on $\mathbb{C}^+$ is at least two-dimensional \cite[Theorem 3]{perelomov1973coherent}:
if $\Lambda$ is a uniqueness set for $A^2_{\alpha} (\mathbb{C}^+)$, then
\begin{align}
\label{eq_nec_per}
\frac{2 \pi}{m_0^+} \leq \frac{4\pi}{\alpha - 1}.
\end{align}
As shown in \cite[Lemma 5.3]{kelly},
\begin{align*}
\frac{2 \pi}{m_0^+} \leq \frac{\vol(G/\Gamma)}{1+ \#P} \leq \vol(G/\Gamma),
\end{align*}
where $\#P$ denotes the number of inequivalent cusps for $\Gamma$.
Thus the necessity of \eqref{eq_compl_sl2r} for cyclicity is stronger than Perelomov's automorphic weight bound for the cyclicity of one specific vector \eqref{eq_nec_per}, but weaker than
Kelly-Lyth's \eqref{eq_kelly_lyth}. Under the assumption that \eqref{eq_nec_per} fails, Perelomov uses certain $\Gamma$-modular forms to construct a non-zero function in $A^2_{\alpha} (\mathbb{C}^+)$ that vanishes on $\Gamma w$. Under the assumption that \eqref{eq_kelly_lyth} fails, Kelly-Lyth also provides such function, by calculating the so-called upper Beurling-Seip density of $\Gamma w$ in terms of the co-volume of $\Gamma$, and
by resorting to Seip's interpolation theorem \cite{seip1993beurling}.
While this article gives a very elementary argument for the necessity of \eqref{eq_compl_sl2r} for the completeness
of $\pi_\alpha(\Gamma)g_w$, we do not
have a similarly simple argument for \eqref{eq_kelly_lyth}.

\subsubsection{Frames and Riesz sequences of reproducing kernels}
\label{sec:frame_kernel}
By Theorem \ref{th_semi}, under \eqref{eq_compl_sl2r}, there exists $g \in A^2_{\alpha} (\mathbb{C}^+)$ such that the orbit $\pi_{\alpha}(\Gamma) g$ is a (Parseval) frame for $A^2_{\alpha} (\mathbb{C}^+)$. In light of Section \ref{sec:perelomov_uniqueness}, it is natural to ask whether the specific choice $g=\repkw$ also provides a frame. Here the answer depends on whether or not $\Gamma$ is co-compact (that is, $G/\Gamma$ is compact). Using \eqref{eq_rep_for}, the frame property reads
\begin{align}\label{eq_samp_1}
A\norm{f}^2_{A^2_{\alpha}}
\leq
\Im(w)^{-\alpha}
\sum_{\gamma \in \Gamma}
\Im(\gamma \cdot w)^\alpha
\abs{f(\gamma \cdot w)}^2
\leq
B \norm{f}^2_{A^2_{\alpha}},
\qquad \qquad f \in A^2_{\alpha} (\mathbb{C}^+),
\end{align}
for some constants $A,B>0$. The stabilizer subgroup $F_w$ is finite because it is simultaneously contained in the discrete set $\Gamma$ and in the compact subgroup $\newg_0 \mathrm{PSO}(2,\mathbb{R}) \newg_0^{-1}$, where $\newg_0 i = w$. Hence, we can rewrite \eqref{eq_samp_1} as a \emph{sampling inequality}:
\begin{align}\label{eq_samp_2}
A'\norm{f}^2_{A^2_{\alpha}}
\leq
\Im(w)^{-\alpha} \# F_w
\sum_{z \in \Gamma w}
\Im(z)^\alpha
\abs{f(z)}^2
\leq
B' \norm{f}^2_{A^2_{\alpha}},
\qquad \qquad f \in A^2_{\alpha} (\mathbb{C}^+).
\end{align}
Based on the characterization of sampling inequalities by Seip \cite{seip1993beurling}, Kelly-Lyth showed that if $\Gamma$ is not co-compact, then $\Gamma w$ never satisfies \eqref{eq_samp_2}, because its so-called lower Beurling-Seip density is zero \cite[p.44]{kelly}. Thus, in this case, $\pi_{\alpha} (\Gamma) \repkw$ fails to be a frame for $A^2_{\alpha} (\mathbb{C}^+)$. On the other hand, if $\Gamma$ is co-compact, the lower Beurling-Seip density of $\Gamma w$ can be computed in term of the co-volume of $\Gamma$ and yields that $\pi_{\alpha} (\Gamma) \repkw$ is a frame for $A^2_{\alpha} (\mathbb{C}^+)$ if and only if \eqref{eq_perelomov} holds, see \cite[p.44]{kelly}.

Similarly, under \eqref{eq_sep_sl2r}, Theorem \ref{th_semi} provides $g \in A^2_{\alpha} (\mathbb{C}^+)$ such that $\pi_{\alpha} (\Gamma) g$ forms a Riesz sequence in $A^2_{\alpha} (\mathbb{C}^+)$, and one may wonder if, under the corresponding strict inequality,
the particular choice $g=\repkw$ is also possible. This is indeed the case if the stabilizer subgroup $F_w$ is trivial: as shown by Kelly-Lyth \cite[Theorem 5.8]{kelly} by invoking Seip's interpolation theorem \cite{seip1993beurling}, the system $\pi_{\alpha} (\Gamma) \repkw$ is a Riesz sequence\footnote{In \cite[Theorem 5.8]{kelly}, it is shown that the orbit $ \Gamma w$ is an \emph{interpolation set} for $A^2_{\alpha} (\mathbb{C}^+)$ if and only if $\vol(G/\Gamma) > \frac{4\pi}{\alpha - 1}$.
It is a standard fact that $\Gamma w$ is an interpolation set if and only if $\pi_{\alpha} (\Gamma) k^{(\alpha)}_w$ is a Riesz sequence; see for example \cite[Section 2.5]{seip2004interpolation} or \cite[Section 3.1]{seip2004interpolation}.} if and only if
\[
\vol(G/\Gamma) > \frac{4\pi}{\alpha - 1}.
\]
If the stabilizer subgroup $F_w$ is non-trivial, then $\pi_{\alpha}(\Gamma) \repkw$ is not a Riesz sequence, because it is linearly dependent (as an indexed set). Indeed, \eqref{eq_rep_for} shows that $\pi_{\alpha}(\gamma) \repkw$ is a multiple of $\repkw$ when $\gamma \in F_w$. To make the problem meaningful, we can eliminate repetitions by considering the reduced orbit
\begin{align*}
\tilde\pi_{\alpha} (\Gamma)\repkw=\left\{\pi_{\alpha} (\gamma)\repkw: \gamma \in \Gamma_w \right\},
\end{align*}
where $\Gamma_w$ is a set of representatives of $\Gamma/F_w$. With this correction, \cite[Theorem 5.8]{kelly} implies that if $\Gamma \subset \mathrm{PSL}(2, \mathbb{R})$
is a Fuchsian group satisfying
\[
\#F_w \vol(G/\Gamma) > \frac{4\pi}{\alpha - 1},
\]
then $\tilde\pi_{\alpha} (\Gamma) \repkw$ is a Riesz sequence in $A^2_{\alpha} (\mathbb{C}^+)$.
Thus, in contrast to the frame property, a Riesz sequence can exist even for lattices that are not co-compact.

\subsubsection{Perelomov's problem with respect to other special vectors}
The functions
\[
\specv (z) = \Big(\frac{z-i}{z+i} \Big)^n (z+i)^{-\alpha},
\qquad n \in \mathbb{N}_0,
\]
form a distinguished orthogonal basis for $A^2_{\alpha} (\mathbb{C}^+)$. Note that $h_0^{(\alpha)}$ is a multiple of the reproducing kernel $k^{(\alpha)}_i \in A^2_{\alpha} (\mathbb{C}^+)$ at $i$ discussed in Section \ref{sec:perelomov_uniqueness}.

In the language of Perelomov \cite{perelomov1972coherent, perelomov1973coherent}, each $\specv$ is a \emph{stationary vector} of the subgroup of rotations $\mathrm{PSO}(2, \mathbb{R})$ in $G = \mathrm{PSL}(2, \mathbb{R})$: for each $r \in \mathrm{PSO}(2, \mathbb{R})$, there exists $\phi_r \in \mathbb{R}$
such that:
\begin{align*}
\pi_{\alpha} (r) \specv = e^{i \phi_r} \specv,
\end{align*}
as a direct calculation shows. Because of stationarity,
given a Fuchsian group $\Gamma \subset G$, the orbit $\pi_{\alpha}(\Gamma) \specv$ can be reduced by selecting for each $\gamma \in \Gamma$ just one representative
modulo $\mathrm{PSO}(2, \mathbb{R})$, the specific choice being immaterial.
The resulting set is a \emph{subsystem of coherent states} in the sense of Perelomov\cite{perelomov1972coherent, perelomov1973coherent}, and it is complete in $A^2_{\alpha} (\mathbb{C}^+)$ if and only if the orbit $\pi_\alpha(\Gamma) \specv$ is.

The coherent state subsystems associated with $\specv$ can be more concretely described as follows \cite{bertrand2002characterization, klauder1994wavelets}.
The subgroup of affine transformations
\begin{align} \label{eq:P}
P := \bigg \{
\newg_{x,y} =
\begin{pmatrix}
\sqrt{y} & x / \sqrt{y} \\
0 & 1/ \sqrt{y}
\end{pmatrix}
:
(x, y) \in \mathbb{R} \times \mathbb{R}^+ \bigg\} \subset \mathrm{PSL}(2,\mathbb{R})
\end{align}
provides representatives for the quotient $G / \mathrm{PSO}(2, \mathbb{R})$, since $G = P \cdot \mathrm{PSO}(2, \mathbb{R})$
and $P \cap \mathrm{PSO}(2, \mathbb{R}) = \{I\}$. In particular, every $\newg \in G$ can be written as
$\newg = \newg_{x,y} r$ for unique $\newg_{x,y} \in P$ and $r \in \mathrm{PSO}(2, \mathbb{R})$.
Recall that $i \in \mathbb{C}^+$ is a fixed point of $\mathrm{PSO}(2, \mathbb{R})$, and, hence,
 $(x,y)$ is $x+iy = \newg_{x,y} \cdot i =\newg \cdot i$. Therefore, the coherent state associated with $\specv$
can be realized as an \emph{affine system}:
\begin{align}\label{eq_per}
\per_\alpha(\specv,\Gamma i) = \left\{ \pi_\alpha(\newg_{x,y}) \specv : x+iy \in \Gamma i \right\}
=
\left\{ y^{-\alpha/2} \specv \big(\tfrac{\cdot-x}{y}\big) : x+iy \in \Gamma i \right\}.
\end{align}
Perelomov's problem concerns the completeness of $\per_\alpha(\specv,\Gamma i)$ in
$A^2_{\alpha} (\mathbb{C}^+)$. While Theorem \ref{th_semi} shows that \eqref{eq_compl_sl2r} is necessary for completeness, we are unaware of literature on corresponding sufficient conditions.

We remark that, as $G$ acts transitively on $\mathbb{C}^+$, the previous conclusions also apply to any other base point $z \in \mathbb{C}^+$ in lieu of $i$. Indeed, if $z=\newg \cdot i$ with $\newg \in G$, then each element of
\begin{align}\label{eq_lala}
\per_\alpha(\specv,\Gamma z) =
\left\{ y^{-\alpha/2} \specv \big(\tfrac{\cdot-x}{y}\big) : x+iy \in \Gamma z \right\}
\end{align}
is a unimodular multiple of an element of
$\pi_\alpha\big(\newg) \per_\alpha\big(\specv,(\newg^{-1}\Gamma \newg) \cdot i\big)$ and vice versa.
Thus, one system is complete if and only if the other is, while
$\vol\big(G / (\newg^{-1}\Gamma \newg)\big)=\vol(G/\Gamma)$. In conclusion, Theorem \ref{th_semi} gives the following:
\begin{align}
\mbox{If the affine system }\eqref{eq_lala}\mbox{ is complete in }
A^2_{\alpha} (\mathbb{C}^+)\mbox{ then } \vol(G/\Gamma) \leq \frac{4\pi}{\alpha - 1}.
\end{align}
The completeness problem can be alternatively reformulated on the real half-line. The connection is provided by the \emph{Paley-Wiener theorem for Bergman spaces} \cite{duren2007paley, sally1967analytic}: the Fourier-Laplace transform
\begin{align*}
\mathcal{F}f (z) = \int_0^\infty f(t) e^{i z t} \, dt, \qquad z \in \mathbb{C}^+,
\end{align*}
is a multiple of an isometric isomorphism between the weighted $L^2$-space
\begin{align*}
L^2(\mathbb{R}^+, t^{-(\alpha-1)}\,dt) =
\left\{ f:\mathbb{R}^+ \to \mathbb{C} \mbox{ measurable} :
\int_{\mathbb{R}^+} |f(t)|^2 t^{-(\alpha-1)} \, dt< \infty \right\}
\end{align*}
and the Bergman space $A^2_{\alpha} (\mathbb{C}^+)$. In addition,
the special vectors $\specv$ correspond under the isomorphism to multiples of
\begin{align} \label{eq:Fourier}
H_n^{(\alpha)} (t) := t^{\alpha - 1} e^{-t} L^{(\alpha - 1)}_n (2t), \qquad t >0,
\end{align}
where $L_n^{\alpha-1}$ is the Laguerre polynomial of degree $n \in \mathbb{N}$ and index $\alpha-1$;
see \cite{duren2007paley}.
The inverse Fourier-Laplace transform thus maps the affine system \eqref{eq_per} into the system
\begin{align}\label{eq_abc}
\mathcal{F}^{-1} \per_\alpha(\specv,\Gamma i)
=\left\{
d^\alpha_n \,y^{-\alpha/2+1} e^{-i \pi x \,\cdot} H^{(\alpha)}_n (y \cdot) \; : \; x+iy \in \Gamma i
\right\},
\end{align}
in $L^2 (\mathbb{R}^+, t^{-(\alpha -1)} dt)$ for a suitable $d^\alpha_n \in \mathbb{C}$. This yields another equivalent formulation of Perelomov's completeness problem.
See also \cite[Section 8.6]{combescure2012coherent}.

With a certain physical motivation, part of Perelomov's work
\cite{perelomov1973coherent} has been adapted to the special vectors $H_n^{(\alpha)}$ by Abreu, Balazs, de Gosson and Mouayn \cite{ABDM}.
Conditionally to the existence of modular forms having certain special properties, and under certain restrictions on $\alpha > 1$, \cite[Corollary 1]{ABDM} asserts that if \eqref{eq_abc} is complete in $L^2(\mathbb{R}^+, t^{-(\alpha-1)}\,dt)$, then
\begin{align}
\label{eq_abdm}
\vol(G/\Gamma) \leq (n+1) \frac{8\pi}{\alpha-1}.
\end{align}
On the other hand, Theorem \ref{th_semi} provides the sharper bound
\begin{align}\label{eq_vvvv}
\vol(G/\Gamma) \leq \frac{4\pi}{\alpha-1},
\end{align}
which is valid without assumptions on the existence of adequate modular forms,
and for all $\alpha>1$.
(Indeed, if
\eqref{eq_abc} is complete in $L^2(\mathbb{R}^+, t^{-(\alpha-1)}\,dt)$
then $[\pi_{\alpha}(\Gamma) \specv]=
[\per_\alpha(\specv,\Gamma)]=A^2_{\alpha} (\mathbb{C}^+)$,
and Theorem \ref{th_semi} gives \eqref{eq_vvvv}.)
\footnote{
The bound stated in \cite[Corollary 1]{ABDM} is
\eqref{eq_abdm} with $\alpha$ instead of $\alpha-1$. We understand this as a  miscalculation
caused by inconsistent normalization of the Bergman space
on \cite[page 352]{ABDM}. The result in \cite{ABDM}
is (equivalently) formulated in terms of the completeness of the system
of functions
$(yt)^{-\alpha/2+1} e^{\pi x i \,t/2} H^{(\alpha)}_n (yt/2)$
within $L^2(\mathbb{R}^+, t^{-1}\,dt)$.
}

\subsection{Heisenberg projective representation and Gabor systems}
\label{sec:gabor}

Let $G = \mathbb{R}^{2d}$.
Define the projective representation $(\pi, L^2 (\mathbb{R}^d))$ through the action
\begin{align} \label{eq:timefreq}
 \pi (z) f (t) = e^{2\pi i y \cdot t} f(t - x), \quad z = (x,\xi) \in \mathbb{R}^{2d}, \; t \in \mathbb{R}^d.
 \end{align}
Then
$
\pi(z + z') = e^{2\pi i \xi' \cdot x } \pi(z) \pi(z')$ for $z = (x, \xi) \in \mathbb{R}^{2d}$ and $z' = (x', \xi') \in \mathbb{R}^{2d}$. Thus the cocycle of $(\pi, L^2 (\mathbb{R}^d))$ is $\sigma (z, z') = e^{2\pi i \xi' \cdot x } \in \mathbb{T}$. Moreover, $\pi$ is is irreducible and square-integrable of formal dimension $d_{\pi} = 1$.
For background, and the appearance of the Heisenberg group, see \cite{folland1989harmonic, groechenig2001foundations}.

Systems of the form $\pi(\Gamma) g$, with $g \in L^2 (\mathbb{R}^d)$
and $\Gamma \subset \mathbb{R}^{2d}$ a lattice, are known as \emph{Gabor systems}
or \emph{Weyl-Heisenberg systems}, and are important in several branches of pure and applied mathematics.
Gabor systems are sometimes also called \emph{canonical} coherent state subsystems in mathematical physics.
The literature on Gabor systems focuses mainly on frames, Riesz sequences, and completeness. Kleppner's condition for a lattice $\Gamma \subseteq \mathbb{R}^{2d}$ and the cocycle $\sigma$ reads: for all $\gamma \in \Gamma \setminus \{0\}$ there exists $\gamma' \in \Gamma$ such that
\[
\sigma(\gamma, \gamma') \overline{\sigma(\gamma, \gamma')} = e^{2\pi i (\gamma_2' \cdot \gamma_1 - \gamma_2 \cdot \gamma_1')} \neq 1.
\]
While for separable lattices $\Gamma = \alpha \mathbb{Z}^d \times \beta \mathbb{Z}^d$, with $\alpha,\beta \in \mathbb{R}$,
Kleppner's condition reduces to $\alpha \beta \not\in \mathbb{Q}$, an explicit characterization of Kleppner's condition for more general lattices is subtle, e.g., see \cite{omland2014primeness,packer1989twisted,han2017note}. Provided that $(\Gamma, \sigma)$
satisfies Kleppner's condition, Theorem \ref{th_main} shows that $\pi(\Gamma)$ admits a frame vector if and only if it admits a complete vector, if and only if
\begin{align} \label{eq:density_gabor}
\vol(G/\Gamma) \leq 1;
\end{align}
while the condition for the existence of a Riesz vector is
\begin{align} \label{eq:density_gabor_2}
\vol(G/\Gamma) \geq 1.
\end{align}
In fact, Theorem \ref{thm:necessary_density} shows that the necessity of the density conditions for completeness, frames, and Riesz sequences holds without assuming Kleppner's condition. Direct proofs of this necessity
go back to Baggett \cite{baggett1990processing}, Daubechies, Landau and Landau \cite{daubechies1995gabor},
and Ramanathan and Steger \cite{ramanathan1995incompletness}, and are also implicitly contained in Rieffel's work \cite{rieffel1981von, rieffel1988projective}.

Our proof of Theorem \ref{thm:necessary_density}
is partially inspired by Janssen's ``classroom proof'' \cite{janssen_density}, which concerns frames and Riesz sequences. Instead of using the frame inequality, as in Proposition \ref{prop:density_quasi},
Janssen uses the so-called \emph{canonical frame expansion} \[f=\sum_{\gamma \in \Gamma} \langle f, \pi(\gamma) S_{g,\Gamma}^{-1} g \rangle \pi(\gamma) g\] associated with a frame $\pi(\Gamma)g$ and frame operator $S_{g,\Gamma}$. The coefficients $\langle f, \pi(\gamma) S_{g,\Gamma}^{-1} g \rangle$ have minimal $\ell^2$ norm among all sequences $c$ such that \[f = \sum_{\gamma \in \Gamma} c_\gamma \pi(\gamma) g\]and this property is leveraged to prove \eqref{eq:density_gabor}. In contrast, we prove Theorem \ref{thm:necessary_density} by resorting to the normalization procedure in Proposition \ref{prop:cyclic_tight-frame}, which applies also to complete systems $\pi(\Gamma)g$ that may not be frames. Similarly, while Janssen treats Riesz sequences $\pi(\Gamma) g$ by invoking properties of the corresponding biorthogonal element $h$ characterized by \[\langle \pi(\gamma) g, \pi(\gamma') h\rangle = \delta_{\gamma', \gamma}, \mbox{ for }\gamma, \gamma' \in \Gamma,\]
and $h \in [\pi(\Gamma)g]$,
we use Proposition \ref{prop:separating_riesz} to reduce the proof to orthonormal sequences, while also treating separating vectors.

As is the case with the necessity of the density conditions, the sufficiency of \eqref{eq:density_gabor} and \eqref{eq:density_gabor_2} for the existence of frames and Riesz vectors also holds without assuming Kleppner's condition. This deep fact, shown by Rieffel \cite{rieffel1981von, rieffel1988projective}, and also a consequence of Bekka's work \cite[Theorem 4]{bekka2004square},
lies beyond the elementary approach presented in this article. Indeed, Rieffel's and Bekka's work require considering not only the operator algebras $\pi(\Gamma)'$ and $\pi(\Gamma)''$, but also certain so-called induced algebras, and in this way fully exploit the coupling theory of von Neumann algebras. We hope that our elementary introduction motivates the reader to delve deeper into operator-algebraic methods. For lattices of the form $\Gamma = A \mathbb{Z}^d \times B \mathbb{Z}^d$,
with $A, B \in \mathrm{GL}(d, \mathbb{R})$, Han and Wang gave a constructive proof of the sufficiency of \eqref{eq:density_gabor} for the existence of frame vectors \cite{han2001lattice}.

\subsubsection{Gaussians and Bargmann-Fock spaces}
The question of choosing specific cyclic or frame vectors has been intensively studied for $d=1$ and lattices in $\mathbb{R}^2$
of the form $\Gamma = \alpha \mathbb{Z} \times \beta \mathbb{Z}$. In his work on foundations of quantum mechanics,
von Neumann \cite{neumann1968mathematische} claimed without proof that
the Gabor system $\pi(\Gamma)g$ generated by the Gaussian function
\begin{align}\label{eq_gauss}
g(t) = 2^{-1/4} e^{-\pi |t|^2}, \qquad t \in \mathbb{R},
\end{align}
is complete in $L^2(\mathbb{R})$ if and only if \eqref{eq:density_gabor} holds.
Proofs of the claim were given by Perelomov \cite{perelomov1971remark}, Bargmann \cite{bargmann1971on}, and Neretin
\cite{neretin2006perelomov}. For rational lattices (i.e., $\alpha \beta \in \mathbb{Q}$), the same claim holds when the Gaussian function is multiplied by a rational function with no real poles \cite{groechenig2016completeness}.

The related question, under which conditions the Gabor system generated by the Gaussian \eqref{eq_gauss} is a frame for $L^2 (\mathbb{R})$ or a Riesz sequence was first considered by Daubechies and Grossmann \cite{daubechies1988frames}, and fully answered independently by Lyubarski\u{\i} \cite{lyubarski1992frames},
and Seip and Wallst\'{e}n \cite{seip1992density, seip1992density2}:
\begin{align*}
\vol(G/\Gamma) < 1,
\end{align*}
is necessary and sufficient for the frame property, while
\begin{align*}
\vol(G/\Gamma) > 1,
\end{align*}
is necessary and sufficient for the Riesz property.

The proofs of Lyubarski\u{\i} \cite{lyubarski1992frames} and Seip-Wallst\'{e}n \cite{seip1992density} work with a $\sigma$-representation
unitarily equivalent  to $(\pi, L^2 (\mathbb{R}))$
on the Bargmann-Fock space $\mathcal{F}^2(\mathbb{C})$ of entire
functions $F:\mathbb{C} \to \mathbb{C}$ having finite norm
\[
\| F \|^2_{\mathcal{F}^2} = \int_{\mathbb{C}} | F(z) |^2 e^{- \pi |z|^2} \; dxdy.
\]
As in Example \ref{sec_bergman}, the distinguished vector $g$ corresponds under the new representation to the reproducing kernel, that is, the vector representing the evaluation functional $F \mapsto F(0)$. A simple proof of the density results was derived by Janssen \cite{janssen1994signal}.

The characterization of the frame and Riesz property for other vectors $g$ is a topic of intense study
\cite{groechenig2014mystery}.

\end{document}